\documentclass[12pt]{amsart}
\usepackage[all,error]{onlyamsmath}
\usepackage{rjmacros-arxiv}

\newcommand{\vp}{\varphi}
\newcommand{\wtcG}{\wtilde{\cG}}
\newcommand{\whf}{\what{f}}
\newcommand{\whg}{\what{g}}
\newcommand{\whh}{\what{h}}
\newcommand{\whgm}{\what{\gamma}}
\newcommand{\whfN}{\what{\fN}}
\newcommand{\wtk}{\wtilde{k}}

\begin{document}
\title[The A-model for finite rank perturbations]{On some
extensions of the A-model}
\author{Rytis Jur\v{s}\.{e}nas}
\address{Vilnius University,
Institute of Theoretical Physics and Astronomy,
Saul\.{e}tekio ave.~3, LT-10257 Vilnius, Lithuania}
\email{Rytis.Jursenas@tfai.vu.lt}
\date{\today}
\begin{abstract}
The A-model for finite rank singular
perturbations of class $\fH_{-m-2}\setm\fH_{-m-1}$,
$m\in\bbN$, is considered from the perspective of
boundary relations. Assuming further that
the Hilbert spaces $(\fH_n)_{n\in\bbZ}$ admit an orthogonal
decomposition $\fH^-_n\op\fH^+_n$, with the corresponding
projections satisfying $P^\pm_{n+1}\subseteq P^\pm_n$,
nontrivial extensions in the A-model are constructed
for the symmetric restrictions in the subspaces.
\end{abstract}
\maketitle
\section{Introduction}
Consider a lower semibounded self-adjoint operator $L$ in
a Hilbert space $\fH_0$. Let $\fH_{n+1}\subseteq\fH_n$,
$n\in\bbZ$, be the
scale of Hilbert spaces associated with $L$.
Let also $\{\vp_\sigma\}$ be the family of linearly independent
functionals of class $\fH_{-m-2}\setm\fH_{-m-1}$, $m\in\bbN$,
where $\sigma$ ranges over an index set $\cS$ of dimension
$d\in\bbN$. Then,
the symmetric restriction $L_{\min}\subseteq L$ to
the domain of $f\in\fH_{m+2}$ such that
$\braket{\vp_\sigma,f}=0$, for all $\sigma$, is an essentially
self-adjoint operator in $\fH_0$. Sequentially, traditional methods,
see \eg \cite{Albeverio00,Hassi09-b}, for describing nontrivial
extensions of $L_{\min}$ (\ie perturbations of $L$)
in $\fH_0$ are insufficient. The classical examples of
higher order singular perturbations are the
point-interactions modeled by the Dirac distribution
and its derivatives.

To construct nontrivial realizations of $L_{\min}$
in Hilbert or Pontryagin spaces,
one considers instead the so-called cascade (A or B) models
\cite{Dijksma05,Dijksma04,Kurasov03,Kurasov03b,Dijksma00} and
the peak model \cite{Jursenas18,Kurasov09}. In these models
the Weyl (or Krein $Q$-) function is the sum of a Nevanlinna
function associated with $L_{\min}$
in $\fH_m$ and a generalized Nevanlinna function associated
with a certain multiplication operator in
a reproducing kernel Pontryagin space
\cite[Theorem~4.10]{Behrndt11}; more on reproducing
kernel spaces can be found in
\cite{Behrndt13,Arlinski08,Behrndt08,Derkach03b}.
Successively, singular perturbations are interpreted
by means
of the compression to the reference space $\fH_0$
of the resolvent of an appropriate extension
in the model space.

Here we study the cascade A-model for rank-$d$
higher order singular perturbations. More precisely,
for a specific choice of model parameters,
we extend the main results
obtained in \cite{Dijksma05} to the case of an arbitrary $d\in\bbN$
(see Theorem~\ref{thm:mn}).
The exposition utilizes the techniques
based on the notion of boundary triples
\cite{Derkach17,Derkach12,Derkach09,Derkach06}.
Then, by assuming
that the Hilbert space $\fH_n$ is expressed as the Hilbert sum
$\fH^-_n\op\fH^+_n$ of its subspaces $\fH^\pm_n$, we examine
nontrivial realizations that account
for the above described Hilbert space decomposition
(Theorem~\ref{thm:mn2}). We assume that the corresponding
orthogonal projections $P^\pm_n$ from $\fH_n$ onto
$\fH^\pm_n$ satisfy the inclusions $P^\pm_{n+1}\subseteq P^\pm_n$.
This further implies that the subspaces $\fH^\pm_n$
reduce the self-adjoint restriction to $\fH_{n+2}$
of $L$ (Theorem~\ref{thm:prn}). As a natural consequence
of our hypothesis is that the Weyl function associated
with the symmetric operator $L_{\min}$ in $\fH_m$ is
the sum of the Weyl functions associated with the
symmetric restrictions to $\fH^\pm_m$ of $L_{\min}$.

The projection of the model to the subspaces just described
has a natural application in quantum mechanics when, for example,
one wishes to account for the contribution to
the eigenvalues of antisymmetric (resp. symmetric) eigenfunctions.
For instance, if one takes $L$ such that $\fH_n=W^n_2\ot\bbC^4$,
where $W^n_2$ is the Sobolev space (Example~\ref{exam}),
then the projections $P^-_n$ and $P^+_n$ onto
the spaces of antisymmetric spin states,
$W^n_2\ot\bbC^1$, and onto the spaces of symmetric spin states,
$W^n_2\ot\bbC^3$, satisfy our hypothesis.
However, a concrete application of the present model
will be demonstrated elsewhere.

Another motivation for considering the A-model, as opposed to
the peak model, arises from an attempt to elude a too restrictive
condition imposed on the Gram matrix
$\cG=(\cG_{\sigma j,\sigma^\prime j^\prime})\in[\bbC^{md}]$
of the peak model; namely, $\cG$ must be diagonal in
$j\in\{1,\ldots,m\}$. Although initially contemplated as an
advantageous feature \cite{Kurasov09}, this restriction is
not satisfied for some operators $L$, for $m>1$, for
a simple reason that the eigenvectors of the
triplet adjoint of $L_{\min}$ for the Hilbert triple
$\fH_m\subseteq\fH_0\subseteq\fH_{-m}$ are not necessarily
orthogonal for distinct eigenvalues
(Example~\ref{exam:peak}).
\section{Preliminaries}
Let $A$ be a densely defined, closed, symmetric operator
in a Pontryagin space $\fH$ (see \eg \cite[Sec.~1.9]{Azizov89})
with an indefinite metric $[\cdot,\cdot]_\fH$.
Let $A^*$ be the adjoint in $\fH$ of $A$.
A triple $(\cH,\Gamma_0,\Gamma_1)$, where
$\cH=(\cH,\braket{\cdot,\cdot}_\cH)$ is a Hilbert
space and $\Gamma\co f\mapsto(\Gamma_0f,\Gamma_1f)$ is
the operator from $\dom A^*$ to $\cH^2(:=\cH\times\cH)$, is
called an ordinary boundary triple (OBT) for $A^*$ if
$\Gamma$ is surjective and the Green identity holds:
\[
[f,g]_{A^*}:=[f,A^*g]_\fH-[A^*f,g]_\fH=
\braket{\Gamma_0f,\Gamma_1g}_\cH-
\braket{\Gamma_1f,\Gamma_0g}_\cH
\]
for all $f,g\in\dom A^*$; see \eg \cite[Definition~2.1]{Derkach95}.
It is shown that
an OBT for $A^*$ in a Pontryagin space (or more generally in
a Krein space) exists iff $A$ admits a self-adjoint extension
in $\fH$ (\cf \cite[Proposition~3.4]{Behrndt11},
\cite[p.~192]{Derkach15}).

If the assumption on the density of $\dom A$ is dropped off,
that is, if $A^*$ is a linear relation
\cite{Hassi09,Hassi07}, then an OBT
$(\cH,\Gamma_0,\Gamma_1)$ for $A^*$ is defined by
considering $\Gamma_i$, $i\in\{0,1\}$, as a mapping
from $A^*$ onto $\cH$. Sequentially, the Green identity
reads
\[
[f,g^\prime]_\fH-[f^\prime,g]_\fH=
\braket{\Gamma_0\whf,\Gamma_1\whg}_\cH-
\braket{\Gamma_1\whf,\Gamma_0\whg}_\cH
\]
for $\whf=(f,f^\prime)$, $\whg=(g,g^\prime)\in A^*$.
The reader may also consult
\cite[Definition~6]{Derkach15}, as well as
\cite[Definition~2.3]{Hassi13},
\cite[Definition~7.11]{Derkach12} in the Hilbert
space case.
In what follows we frequently identify operators with
their graphs. Then the present definition of an OBT
reduces to the previous definition as long as
$A$ becomes densely defined.

A proper extension $A_\Theta$ of $A$, \ie such that
$A\subseteq A_\Theta\subseteq A^*$, is uniquely determined
by a linear relation $\Theta$ in $\cH$ via
$\Theta=\Gamma A_\Theta$ with
$A_\Theta=\{\whf\in A^*\vrt\Gamma\whf\in\Theta\}$; see \eg
\cite[Proposition~2]{Derkach15},
\cite[Proposition~2.5]{Hassi13},
\cite[Proposition~7.12]{Derkach12},
\cite[Proposition~2.1]{Derkach95}.
In particular, a distinguished self-adjoint extension
$A_0:=A^*\vrt_{\ker\Gamma_0}$ corresponds to a self-adjoint
linear relation $\Theta=\{0\}\times\cH$ (and similarly for
the transversal one, corresponding to $\Theta=\cH\times\{0\}$).
A self-adjoint linear relation in a Krein
(or Pontryagin) space may have an empty resolvent set
(see \eg \cite[Example~3.7]{Behrndt11}). However, if
there exists at least one self-adjoint extension of $A$,
say $\wtA$, whose resolvent set $\res\wtA$ is nonempty, then
there exists an OBT for $A^*$ such that $\wtA=A_0$.

Let $A$ be a closed symmetric operator as above.
Let $\fN_z(A^*):=\ker(A^*-z)$, $z\in\bbC$, denote the
eigenspace of a linear relation $A^*$ (and similarly for
other linear relations and operators). Let
$\whfN_z(A^*)$ be the set of the pairs $(f_z,zf_z)$ with
$f_z\in\fN_z(A^*)$. Let also $\pi_1$ denote the orthogonal
projection in the Hilbert sum of a Hilbert space with
itself onto the first factor. Assume that
the resolvent set $\res A_0\neq\emptyset$.
The $\gamma$-field $\gamma$
and the Weyl function $M$ corresponding to
the OBT $(\cH,\Gamma_0,\Gamma_1)$ for $A^*$ are bounded
operator valued functions defined by
\cite[Definition~7]{Derkach15},
\cite[Definition~2.6]{Hassi13}
\[
\gamma(z):=\pi_1\whgm(z)\,,\quad
\whgm(z):=(\Gamma_0\vrt_{\whfN_z(A^*)})^{-1}\,,
\quad M(z):=\Gamma_1\whgm(z)
\]
for $z\in\res A_0$.
Then the resolvent of a
closed proper extension $A_\Theta$, \ie such that
$\Theta$ is closed, is represented by the Krein--Naimark
resolvent formula (see \eg
\cite[Theorem~4]{Derkach15},
\cite[Theorem~2.1]{Derkach95})
\[
(A_\Theta-z)^{-1}=(A_0-z)^{-1}+\gamma(z)
(\Theta-M(z))^{-1}\gamma(\ol{z})^*
\]
for $z\in\res A_0\mcap\res A_\Theta$.
Moreover, $z\in\res A_\Theta$ iff
$0\in\res(\Theta-M(z))$.

Let $\fH=(\fH,[\cdot,\cdot]_\fH)$ be a Krein
(or in particular Pontryagin) space,
let $\cH=(\cH,\braket{\cdot,\cdot}_\cH)$ be a Hilbert space.
Consider a linear relation
$\Gamma\subseteq\fH^2\times\cH^2$.
Let $\Gamma^{[+]}$ be its Krein space adjoint:
\begin{align*}
\Gamma^{[+]}:=&\{((h_\circ,h^\prime_\circ),(g,g^\prime))
\in\cH^2\times\fH^2\vrt
(\forall ((f,f^\prime),(h,h^\prime))\in\Gamma)
\\
&[f,g^\prime]_\fH-[f^\prime,g]_\fH=
\braket{h,h^\prime_\circ}_\cH-
\braket{h^\prime,h_\circ}_\cH\}\,.
\end{align*}
Then $\Gamma$ is said to be an isometric (resp. unitary)
linear relation if the inverse linear relation
$\Gamma^{-1}\subseteq\Gamma^{[+]}$
(resp. $\Gamma^{-1}=\Gamma^{[+]}$). If $\Gamma$ is unitary
and additionally single-valued (\ie an operator
identified with its graph), then
by \cite[Corollary~2.4(i)]{Derkach06} $\ol{\ran}\Gamma=\cH^2$
(the closure of the range). If, moreover,
$\dom\Gamma$ is closed, then also $\ran\Gamma$ is closed,
and is given by $\ran\Gamma=\cH^2$
(\cite[Corollary~2.4(iii)]{Derkach06}).

Throughout we use quite standard notation
for the domain $\dom A$, the range $\ran A$,
the kernel $\ker A$, and the multivalued part
$\mul A$ of a linear relation $A$. The resolvent
set of $A$ is denoted by $\res A$, the point
spectrum by $\sigma_p(A)$.
\section{The A-model for finite rank perturbations}
Let $\fH_{n+1}\subseteq\fH_n$, $n\in\bbZ$, be the scale
of Hilbert spaces associated with a lower semibounded
self-adjoint operator $L$ defined
in the reference Hilbert space $\fH_0$ with domain
$\dom L=\fH_2$.
The scalar product in $\fH_n$ is defined via the
scalar product $\braket{\cdot,\cdot}_0$ in $\fH_0$
by scaling according to
\[
\braket{\cdot,\cdot}_n:=
\braket{b_n(L)^{1/2}\cdot,b_n(L)^{1/2}\cdot}_0\,,\quad
b_n(L):=(L-z_1)^n\,.
\]
The number $z_1\in\res L\mcap\bbR$ is fixed and
referred to as the model parameter. Let us mention that
the above definition of the $\fH_n$-scalar product
allows us to avoid extra technicalities arising
when, for example, one chooses $b_n(L)$ as the product
of $(L-z_j)$ for $j\in\{1,\ldots,n\}$ for not
necessarily identical model parameters $z_j$, as is done
in \cite{Dijksma05} (where $z_j=-a_j$), or when, on top
of that, one assumes $L$ not necessarily semibounded, in which
case one should put $\abs{L}$ in $b_n(L)$ instead of $L$.
On the other hand, our definition of the scalar product
predefines the inner structure of the model space
(to be defined later); namely, it is shown in
\cite[Theorem~3.2(iii)]{Dijksma05} for $d=1$
that the present choice of the model parameters
(\ie $a_j=-z_1$ for all $j$) leads to an indefinite
inner product space, as the model space.
Let us moreover advertise that the current definition
of the unitary operator $b_n(L)^{1/2}$ (from $\fH_n$
to $\fH_0$) is not allowed in the peak model
\cite{Kurasov09}, which is a purely Hilbert space
model (\cf \cite[Theorem~3.2(ii)]{Dijksma05}).

To $L=L_0$ one associates an operator $L_n:=L\vrt_{\fH_{n+2}}$
in $\fH_n$. Then $L_n$ is self-adjoint in $\fH_n$, and moreover
$L_{n+1}\subset L_n$ and $\res L_n=\res L$
(\cf Section~\ref{sec:Lnpm}).
For notational simplicity
we drop-off the subscript when no confusion can arise.

Let us fix $m\in\bbN$.
Let $L_{\max}$ denote the triplet adjoint of $L_{\min}$
for the Hilbert triple $\fH_m\subset\fH_0\subset\fH_{-m}$;
see also \cite[Theorem~2.1]{Dijksma05},
\cite[Definition~3.1]{Kurasov09},
\cite[Proposition~4.2]{Jursenas18}.
The operator $L_{\max}$ extends $L_{-m+2}$ to
\[
\dom(L_{\max})=\fH_{-m+2}\dsum\fN_z(L_{\max})\,,
\quad z\in\res L
\]
(direct sum). $\fN_z(L_{\max})$
is the linear span of the singular elements
$\{g_\sigma(z)\in\fH_{-m}\setm\fH_{-m+1}\}$,
each being defined so that
$b_m(L)^{-1}g_\sigma(z)\in\fH_m\setm\fH_{m+1}$
is a deficiency element of the adjoint $L^*_{\min}$ in $\fH_m$
of a densely defined, closed, symmetric operator
$L_{\min}$ in $\fH_m$ with defect numbers $(d,d)$.
Let us recall that the domain of $L_{\min}$ is
parametrized via the family of linearly independent
functionals $\{\vp_\sigma\in\fH_{-m-2}\setm\fH_{-m-1}\}$
according to $\braket{\vp_\sigma,f}=0$ for
$f\in\fH_{m+2}$; the duality pairing $\braket{\cdot,\cdot}$
between $\fH_{-m-2}$ and $\fH_{m+2}$ is defined
via the $\fH_0$-scalar product in a usual way
(\cf \cite[Eq.~(1.17)]{Albeverio00}). In the sequel
we also use the vector notation
$\braket{\vp,\cdot}=(\braket{\vp_\sigma,\cdot})
\co\fH_{m+2}\lto\bbC^d$, and similarly for
other duality pairings.
In terms of
the functionals $\{\vp_\sigma\}$ the eigenvectors
of $L_{\max}$ are then given (in the generalized sense)
by $g_\sigma(z):=(L-z)^{-1}\vp_\sigma$.

As the space $\fH_{-m}$ in which $L_{\max}$ acts is too
large, following the lines of \cite{Dijksma05}
one further considers $L_{\max}$ in a finite-dimensional
extension of $\fH_m$, referred to as an intermediate
(or model) space. We now discuss the construction of
the space in more detail.

Consider an $md$-dimensional linear space
\[
\fK_{\mrm{A}}:=\spn\{h_\alpha\vrt
\alpha=(\sigma,j)\in\cS\times J\}\,,
\quad
J:=\{1,2,\ldots,m\}
\]
($\cS$ is an index set of dimension $d$)
spanned by the elements
\[
h_{\sigma j}:=(L-z_1)^{-j}\vp_\sigma\in
\fH_{-m-2+2j}\setm\fH_{-m-1+2j}\,.
\]
Note that $h_{\sigma 1}=g_\sigma(z_1)\in\fN_{z_1}(L_{\max})$.
An element $k\in\fK_{\mrm{A}}\subseteq\fH_{-m}$ is thus of the form
\[
k=\sum_\alpha d_\alpha(k)
h_\alpha\,,\quad d_\alpha(k)\in\bbC\,.
\]
Since the system $\{h_\alpha\}$ is linearly independent,
the Gram matrix
\[
\wtcG_{\mrm{A}}=([\wtcG_{\mrm{A}}]_{\alpha\alpha^\prime})
\in[\bbC^{md}]\,,\quad
[\wtcG_{\mrm{A}}]_{\alpha\alpha^\prime}:=
\braket{h_\alpha,h_{\alpha^\prime}}_{-m}
\]
is positive definite, and one establishes a bijective correspondence
$\fK_{\mrm{A}}\ni k\leftrightarrow d(k)=(d_\alpha(k))\in\bbC^{md}$.
Observe that $\fK_{\mrm{A}}\mcap\fH_{m-1}=\{0\}$.

Define a linear space
\[
\cH_{\mrm{A}}:=(\fH_m\dsum\fK_{\mrm{A}},[\cdot,\cdot]_{\mrm{A}})
\]
with an indefinite metric
\[
[f+k,f^\prime+k^\prime]_{\mrm{A}}:=
\braket{f,f^\prime}_m+
\braket{d(k),\cG_{\mrm{A}}d(k^\prime)}_{\bbC^{md}}
\]
for $f,f^\prime\in\fH_m$; $k,k^\prime\in\fK_{\mrm{A}}$.
An Hermitian matrix
$\cG_{\mrm{A}}=([\cG_{\mrm{A}}]_{\alpha\alpha^\prime})
\in[\bbC^{md}]$ is referred to as the
Gram matrix of the A-model.
The model space $\cH_{\mrm{A}}$ is a Hilbert space
if $\cG_{\mrm{A}}\geq0$ and a Pontryagin space otherwise.
Let also
\[
\cH^\prime_{\mrm{A}}:=(\fH_m\op\bbC^{md},
[\cdot,\cdot]^\prime_{\mrm{A}})
\]
with an indefinite metric
\[
[(f,\xi),(f^\prime,\xi^\prime)]^\prime_{\mrm{A}}:=
\braket{f,f^\prime}_m+
\braket{\xi,\cG_{\mrm{A}}\xi^\prime}_{\bbC^{md}}
\]
for $(f,\xi)$, $(f^\prime,\xi^\prime)\in\fH_m\op\bbC^{md}$.
The isometric isomorphism (unitary operator) from
$\cH_{\mrm{A}}$ onto $\cH^\prime_{\mrm{A}}$,
realized via the above established
bijective correspondence $\fK_{\mrm{A}}\leftrightarrow\bbC^{md}$,
is denoted by $U_{\mrm{A}}$.

The construction of nontrivial extensions to
$\cH_{\mrm{A}}$ of $L_{\min}$ relies upon the following lemma;
\cf \cite[Eq.~(2.3)]{Dijksma05}.
\begin{lem}\label{lem:1}
The restriction to $\cH_{\mrm{A}}$
of $L_{\max}$ is the operator $A_{\max}$ given by
\begin{align*}
\dom A_{\max}=&
\{f^\#+h_{m+1}(c)+k\vrt f^\#\in\fH_{m+2}\,;\,
k\in\fK_{\mrm{A}}\,;
\\
&h_{m+1}(c):=\sum_\sigma c_\sigma h_{\sigma,m+1}\,;\,
c=(c_\sigma)\in\bbC^d\,;
\\
&h_{\sigma,m+1}:=b_{m+1}(L)^{-1}\vp_\sigma\in
\fH_{m}\setm\fH_{m+1}\}\,,
\\
A_{\max}(f^\#+h_{m+1}(c)+k)=&
Lf^\#+z_1h_{m+1}(c)+\wtk\,,
\quad \wtk\in\fK_{\mrm{A}}\,,
\\
d(\wtk):=&\fM_dd(k)+\eta(c)\,,\quad
\eta(c):=(\delta_{jm}c_\sigma)\in\bbC^{md}
\end{align*}
where the matrix $\fM_d:=\fM\op\cdots\op \fM$
($d$ times)
is the matrix direct sum of $d$ matrices
$\fM=(\fM_{jj^\prime})\in[\bbC^m]$ defined by
\[
	\fM_{jj^\prime}:=\delta_{jj^\prime}z_1+
	\delta_{j+1,j^\prime}\,,\quad j\in J\setm\{m\}\,,
	\quad j^\prime\in J
\]
and $\fM_{mj^\prime}:=\delta_{j^\prime m}z_1$, $j^\prime\in J$.
For $m=1$ one puts $\fM:=z_1$.
\end{lem}
\begin{proof}
By definition, the action of $L_{\max}$ on
$f+k\in\fH_m\dsum\fK_{\mrm{A}}$ is given
(in the generalized sense) by
\begin{align*}
L_{\max}(f+k)=&Lf+\sum_\sigma z_1d_{\sigma 1}(k)h_{\sigma 1}
+\sum_\sigma\sum_{j=2}^md_{\sigma j}(k)
L(L-z_1)^{-j}\vp_\sigma
\\
=&Lf+z_1k+\sum_\sigma\sum_{j=1}^{m-1}d_{\sigma,j+1}(k)
h_{\sigma j}\,.
\end{align*}
Now $Lf\in\fH_{m-2}$, thus the range restriction
$L_{\max}(f+k)\in\fH_m\dsum\fK_{\mrm{A}}$ implies that
$f$ is of the form $f^\#+g$ for some $f^\#\in\fH_{m+2}$
and $g\in\fH_m$ such that $Lg\in\cH_{\mrm{A}}$. By noting
that $Lh_{m+1}(c)=z_1h_{m+1}(c)+h_m(c)$
($h_m(c)\in\fK_{\mrm{A}}$ is defined similar to
$h_{m+1}(c)$) for an arbitrary $c\in\bbC^d$, one concludes that
$g=h_{m+1}(c)$, and the required result follows.
\end{proof}
Now we state the main realization theorem in the A-model.
\begin{thm}\label{thm:mn}
Assume that an invertible Hermitian matrix $\cG_{\mrm{A}}$
satisfies the commutation relation
\begin{equation}
\cG_{\mrm{A}}\fM_d=\fM^*_d\cG_{\mrm{A}}\,.
\label{eq:hypho2}
\end{equation}
Then the triple $(\bbC^d,\Gamma^{\mrm{A}}_0,\Gamma^{\mrm{A}}_1)$,
where
$\Gamma^{\mrm{A}}\co f\lmap(\Gamma^{\mrm{A}}_0f,\Gamma^{\mrm{A}}_1f)$
from $\dom A_{\max}$ to $\bbC^d\times\bbC^d$ is
defined by
\begin{align*}
\Gamma^{\mrm{A}}_0(f^\#+h_{m+1}(c)+k):=&c\,,
\\
\Gamma^{\mrm{A}}_1(f^\#+h_{m+1}(c)+k):=&
\braket{\vp,f^\#}-[\cG_{\mrm{A}}d(k)]_m
\end{align*}
with
\[
[\cG_{\mrm{A}}d(k)]_m:=([\cG_{\mrm{A}}d(k)]_{\sigma m})
\in\bbC^d
\]
and $f^\#\in\fH_{m+2}$, $k\in\fK_{\mrm{A}}$, $c\in\bbC^d$,
is an OBT for the adjoint
$A^*_{\min}=A_{\max}$ of a densely defined, closed,
symmetric operator
$A_{\min}=A_{\max}\vrt_{\ker \Gamma^{\mrm{A}}}$ in $\cH_{\mrm{A}}$.

Moreover,
for a (closed) linear relation $\Theta$ in $\bbC^d$,
a proper extension $A_\Theta$ of $A_{\min}$ is the restriction
of $A_{\max}$ to the set of $f\in\dom A_{\max}$ such that
$\Gamma^{\mrm{A}}f\in\Theta$. The Krein--Naimark resolvent
formula reads
\[
(A_\Theta-z)^{-1}=
(A_0-z)^{-1}+\gamma_{\mrm{A}}(z)
(\Theta-M_{\mrm{A}}(z))^{-1}\gamma_{\mrm{A}}(\ol{z})^*
\]
for $z\in\res A_0\mcap \res A_\Theta$. The
resolvent of a distinguished self-adjoint extension
$A_0:=A_{\{0\}\times\bbC^d}$ is given by
\[
(A_0-z)^{-1}=U^*_{\mrm{A}}[(L-z)^{-1}\op
(\fM_d-z)^{-1}]U_{\mrm{A}}
\]
for $z\in\res A_0=\res L\setm\{z_1\}$.
The $\gamma$-field $\gamma_{\mrm{A}}$ and the Weyl function
$M_{\mrm{A}}$ corresponding to
$(\bbC^d,\Gamma^{\mrm{A}}_0,\Gamma^{\mrm{A}}_1)$
are given by
\[
\gamma_{\mrm{A}}(z)\bbC^d=\fN_z(A_{\max})
=\{\sum_\sigma c_\sigma F_\sigma(z)\vrt
c_\sigma\in\bbC\}\,,\quad
F_\sigma(z):=\frac{g_\sigma(z)}{(z-z_1)^m}
\]
and
\[
M_{\mrm{A}}(z)=q(z)+r(z)\quad\text{on}\quad\bbC^d
\]
for $z\in\res A_0$. The Krein $Q$-function $q$
of $L_{\min}$ is defined by
\[
q(z)=([q(z)]_{\sigma\sigma^\prime})\in[\bbC^d]\,,\quad
[q(z)]_{\sigma\sigma^\prime}:=
(z-z_1)\braket{\vp_\sigma,(L-z)^{-1}
h_{\sigma^\prime,m+1}}
\]
for $z\in\res L$,
and the generalized Nevanlinna function $r$ is defined by
\[
r(z)=([r(z)]_{\sigma\sigma^\prime})\in[\bbC^d]\,,\quad
[r(z)]_{\sigma\sigma^\prime}:=-\sum_j
\frac{[\cG_{\mrm{A}}]_{\sigma m,\sigma^\prime j}}{
(z-z_1)^{m-j+1}}
\]
for $z\in\bbC\setm\{z_1\}$.
\end{thm}
\begin{proof}
By Lemma~\ref{lem:1},
the boundary form of $A_{\max}$
is given by
\[
[f,g]_{A_{\max}}=
\braket{d(k),(\cG_\fM-\cG^*_\fM)d(k^\prime)}_{\bbC^{md}}
+\braket{\Gamma^{\mrm{A}}_0f,\Gamma^{\mrm{A}}_1g}_{\bbC^d}
-\braket{\Gamma^{\mrm{A}}_1f,\Gamma^{\mrm{A}}_0g}_{\bbC^d}
\]
with $\cG_\fM:=\cG_{\mrm{A}}\fM_d$, where
$f=f^\#+h_{m+1}(c)+k\in\dom A_{\max}$;
$g=g^\#+h_{m+1}(c^\prime)+k^\prime\in\dom A_{\max}$;
$f^\#,g^\#\in\fH_{m+2}$; $c,c^\prime\in\bbC^d$;
$k,k^\prime\in\fK_{\mrm{A}}$.
Assuming that
\[
\ker\cG_{\mrm{A}}=\{0\}\quad\text{and}\quad
\fM^*_d\cG_{\mrm{A}}\bbC^{md}\subseteq\ran\cG_{\mrm{A}}
\]
the adjoint $A_{\min}:=A^*_{\max}$
in $\cH_{\mrm{A}}$ is given by
\begin{align*}
\dom A_{\min}=&\ker\Gamma^{\mrm{A}}\,,
\\
A_{\min}(f^\#+k)=&
Lf^\#+\sum_\alpha
[\cG^{-1}_{\mrm{A}}\fM^*_d\cG_{\mrm{A}}d(k)]_\alpha
h_\alpha
\end{align*}
and hence the boundary form of $A_{\min}$ reads
\[
[f,g]_{A_{\min}}=
\braket{d(k),(\cG^*_\fM-\cG_\fM)d(k^\prime)}_{\bbC^{md}}
\]
with $f=f^\#+k\in\dom A_{\min}$ and
$g=g^\#+k^\prime\in\dom A_{\min}$ as above.
One verifies that the adjoint $A^*_{\min}=A_{\max}$,
and hence $A_{\max}$ is closed in $\cH_{\mrm{A}}$.

If \eqref{eq:hypho2} holds,
the boundary form of $A^*_{\min}$
satisfies an abstract Green identity. Thus,
since $\Gamma^{\mrm{A}}$ is single-valued and surjective,
the triple $(\bbC^d,\Gamma^{\mrm{A}}_0,\Gamma^{\mrm{A}}_1)$
is an OBT for $A^*_{\min}$.

The eigenvalue equation for $A_{\max}$ yields
\begin{equation}
f^\#=(z-z_1)(L-z)^{-1}h_{m+1}(c)\,,\quad
d(k)=-(\fM_d-z)^{-1}\eta(c)
\label{eq:x}
\end{equation}
for $f^\#+h_{m+1}(c)+k\in\dom A_{\max}$ as above.
Now
\[
[(\fM_d-z)^{-1}\eta(c)]_{\sigma j}=\sum_{\sigma^\prime}
[(\fM_d-z)^{-1}]_{\sigma j,\sigma^\prime m}c_{\sigma^\prime}
\]
with $c=(c_\sigma)\in\bbC^d$ and with
\[
[(\fM_d-z)^{-1}]_{\sigma j,\sigma^\prime m}=
\delta_{\sigma\sigma^\prime}
[(\fM-z)^{-1}]_{jm}\,,\quad
[(\fM-z)^{-1}]_{jm}=\frac{-1}{(z-z_1)^{m-j+1}}\,.
\]
Thus, by noting that
\[
(L-z)^{-1}(L-z_1)^{-m}+\sum_j
(L-z_1)^{-j}(z-z_1)^{-m+j-1}=(L-z)^{-1}(z-z_1)^{-m}
\]
one concludes that the eigenvector
$f^\#+h_{m+1}(c)+k\in\fN_z(A_{\max})$
is given as stated in the theorem.

Finally, the Weyl function
\[
M_{\mrm{A}}(z)c=\braket{\vp,f^\#}-[\cG_{\mrm{A}}d(k)]_m
\]
for $f^\#$ and $k$ as in \eqref{eq:x}.
The first term on the right-hand side
defines $q(z)c$ and the second term defines $r(z)c$.
\end{proof}
Let us mention that the $Q$-function $q$ is actually
the Weyl function associated with
a certain boundary triple for the adjoint $L^*_{\min}$
in $\fH_m$; see Corollary~\ref{cor:mn2x} below.
While $q$ is a Nevanlinna function,
$r$ is a generalized Nevanlinna function, and the Nevanlinna
class \cite{Behrndt13,Arlinski08} depends on the particular
choice of the Gram matrix $\cG_{\mrm{A}}$.

The matrix $\cG_\fM:=\cG_{\mrm{A}}\fM_d$ is Hermitian iff
\begin{equation}
[\cG_{\mrm{A}}]_{\sigma j,\sigma^\prime j^\prime}=0\,,\quad
[\cG_{\mrm{A}}]_{\sigma j,\sigma^\prime m}=
\ol{[\cG_{\mrm{A}}]_{\sigma^\prime m,\sigma j}}=
[\cG_{\mrm{A}}]_{\sigma,j+1;\sigma^\prime, m-1}\,,
\quad j,j^\prime\in J\setm\{m\}
\label{eq:GAcomm}
\end{equation}
for $m\geq2$.
For $m=1$, however, the matrix
$\cG_\fM=z_1\cG_{\mrm{A}}$ is automatically Hermitian.

Due to \eqref{eq:GAcomm}, several remarks are in order.
First one verifies that $r$ is symmetric with respect to
the real axis, that is,
$r(z)^*=r(\ol{z})$,
because $[\cG_{\mrm{A}}]_{\sigma m,\sigma^\prime j}=
[\cG_{\mrm{A}}]_{\sigma j,\sigma^\prime m}$
($j\in J$) by \eqref{eq:GAcomm}.
Note that $q(z)^*=q(\ol{z})$
is clear from the definition.
Next, one observes that the Gram matrix $\wtcG_{\mrm{A}}$ does
not satisfy \eqref{eq:hypho2} for $m\geq2$, because
$[\wtcG_{\mrm{A}}]_{\sigma 1,\sigma 1}>0$.
This shows that, in order use Theorem~\ref{thm:mn}
for $m\geq2$,
one cannot define the Gram matrix of the A-model
in a way that is done in the peak model.
\begin{rem}
Let us recall that in the peak model the parameters
$\{a_j\}$ are all necessarily distinct. However, putting
$a_j=-z_1+\delta_{j-1}$ for $\delta_j\neq0$ and $j\in J\setm\{1\}$
and $m\geq2$, and formally taking the limits
$\delta_j\lto\delta_{j-1}$, as well as $\delta_1\lto0$,
one can show by induction that the $Q$-function associated
with the Gram matrix $\cG$ of the peak model approaches $r$,
up to $O(\delta_1)$, with
$[\cG_{\mrm{A}}]_{\sigma m,\sigma^\prime j}=
[\wtcG_{\mrm{A}}]_{\sigma m,\sigma^\prime j}$.
Notice that $[\wtcG_{\mrm{A}}]_{\sigma m,\sigma^\prime j}$,
with $m\geq2$, satisfies the second relation in \eqref{eq:GAcomm}.
On the other hand, taking the above described limits,
the matrix element $\cG_{\sigma 1,\sigma^\prime 2}=
[\wtcG_{\mrm{A}}]_{\sigma 1,\sigma^\prime 1}+O(\delta_1)$,
so the requirement that $\cG$ must be diagonal in
$j$---which is essential in applying the extension theory of
symmetric operators in the peak model---fails for $m\geq2$.
For $m=1$, both models produce the same
Nevanlinna function $r(z)=\cG_{\mrm{A}}/(z_1-z)$, provided that
$\cG_{\mrm{A}}=\wtcG_{\mrm{A}}(\in[\bbC^d])$.
\end{rem}
\begin{exam}\label{exam:peak}
We briefly demonstrate by a concrete example
the case when the eigenvectors $\{g_\sigma(z)\}$
of $L_{\max}$ are not orthogonal for distinct $z$,
that is, the example when the peak model cannot be
applied. We consider the two-particle Rashba
spin-orbit-coupled operator $L$ in $\fH_0=\mrm{L}^2(\bbR^6)\ot\bbC^4$
with point-interaction between the two cold atoms
\cite{Jursenas18b}. The operator is nonseparable
in the center-of-mass coordinate system
$(x,X)\in\bbR^3\times\bbR^3$ ($x$ is the distance between the
two atoms, $X$ is the center-of-mass coordinate) for
a nonzero spin-orbit-coupling strength $\varepsilon$.
The interaction is modeled
by the Dirac distribution $\vp_\sigma\in\fH_{-4}\setm\fH_{-3}$
concentrated at $x=0$:
$\braket{\vp_\sigma,f}=N_\sigma f_\sigma(0,X)$,
$f=\sum_\sigma f_\sigma\ot\ket{\sigma}\in\fH_4$,
$N_\sigma>0$ is the normalization constant,
$\{\ket{\sigma}\}$ is an orthonormal basis of
$\bbC^4$. Thus we have $m=2$ and $d=4$. For simplicity,
we assume that $\varepsilon$ is negligibly small.
In this regime $L$ approximates, up to $O(\varepsilon)$,
the operator $(-2\Delta_x-\frac{1}{2}\Delta_X)\ot I_{\bbC^4}$
(\cf \cite[Eq.~(8)]{Albeverio01}),
where $\Delta_x$ (resp. $\Delta_X$) is the Laplacian
in $x\in\bbR^3$ (resp. $X\in\bbR^3$). Then the
distribution $g_\sigma(z)\in\fH_{-2}\setm\fH_{-1}$
admits a relatively simple form
\[
g_\sigma(z)=-\frac{N_\sigma}{(2\pi)^3}
\frac{z K_2(\abs{\cdot-W_0}\sqrt{-z})}{\abs{\cdot-W_0}^2}
\ot\ket{\sigma}\,,\quad W_0=(0,X)\,,\quad
z\in\bbC\setm[0,\infty]
\]
where $K_2$ is the Macdonald function of second order.
Because $m=2$, it suffices to have in the (peak) model
two distinct model parameters $z_1$, $z_2<0$ (or else
$a_1$, $a_2>0$). Because now $b_2(L)=(L-z_1)(L-z_2)$,
the Gram matrix element $\cG_{\sigma 1,\sigma 2}$ reads
\begin{align*}
\cG_{\sigma 1,\sigma 2}:=&
\braket{g_\sigma(z_1),g_\sigma(z_2)}_{-2}=
\braket{g_\sigma(z_1),b_2(L)^{-1}g_\sigma(z_2)}_{0}
\\
=&\braket{\vp_\sigma,[(L-z_1)(L-z_2)]^{-2}\vp_\sigma}
=\braket{\vp_\sigma,\frac{\pd^2}{\pd u\,\pd v}
[(L-u)(L-v)]^{-1}\vp_\sigma\vrt_{u=z_1\,;\;v=z_2} }
\\
=&\braket{\vp_\sigma,\frac{\pd^2}{\pd u\,\pd v}
\frac{g_\sigma(u)-g_\sigma(v)}{u-v}\vrt_{u=z_1\,;\;v=z_2} }
\\
=&-\frac{N^2_\sigma}{(2\pi)^3}\lim_{r\lto0}
\frac{1}{r^2}\frac{\pd^2}{\pd u\,\pd v}
\frac{uK_2(r\sqrt{-u})-vK_2(r\sqrt{-v})}{u-v}
\vrt_{u=-a_1\,;\;v=-a_2}
\\
=&\frac{N^2_\sigma}{(2\pi)^3 2^4}
\frac{2a_1a_2\log(a_1/a_2)-a^2_1+a^2_2}{(a_2-a_1)^3}
\end{align*}
up to $O(\varepsilon^2)$
(a more accurate computation of $\cG_{\sigma 1,\sigma 2}$
shows that the term $O(\varepsilon)$ vanishes).
\end{exam}
\section{Projections}
In the remaining part of the present paper
we develop the A-model in the subspaces
\[
\cH^{\prime\,-}_{\mrm{A}}:=
(\fH^-_m\op\bbC^{md},[\cdot,\cdot]^\prime_{\mrm{A}})\,,
\quad
\cH^{\prime\,+}_{\mrm{A}}:=
(\fH^+_m\op\bbC^{md},[\cdot,\cdot]^\prime_{\mrm{A}})
\]
of $\cH^\prime_{\mrm{A}}$, by assuming that
the Hilbert space $\fH_m=\fH^-_m\op\fH^+_m$
is the Hilbert (orthogonal) sum of its subspaces
$\fH^\pm_m$. The analogue of Theorem~\ref{thm:mn},
in the case when $\fH^\pm_{n+1}\subseteq\fH^\pm_n$
($\forall n\in\bbZ$) densely, is stated in Theorem~\ref{thm:mn2}.
First we discuss the properties of the projections that we use
later on, then we consider the restrictions to $\fH^\pm_n$
of $L_n$, and then finally we describe the min-max
operators defined in $\cH^{\prime\,\pm}_{\mrm{A}}$.
The principal difference between the case of
the minimal operator
$A_{\min}$ considered in $\cH_{\mrm{A}}$ and its
analogue $A^-_{\min}$ (resp. $A^+_{\min}$) considered in
$\cH^{\prime\,-}_{\mrm{A}}$ (resp. $\cH^{\prime\,+}_{\mrm{A}}$)
is that $A^-_{\min}$ (resp. $A^+_{\min}$) becomes
nondensely defined in general, that is, the corresponding
maximal operator $A^-_{\max}$ (resp. $A^+_{\max}$)
is a linear relation.

Let $P^-_n$ be an orthogonal projection in $\fH_n$
onto a subspace $\fH^-_n\subseteq\fH_n$ and let
$P^+_n:=I_{\fH_n}-P^-_n$, an orthogonal projection
in $\fH_n$ onto $\fH^+_n:=(\fH^-_n)^{\bot_{\fH_n}}$.
Here and elsewhere the subscript in $\bot_{\fH_n}$ indicates
with respect to which Hilbert space one takes the orthogonal
complement.
\begin{lem}
$P^-_n$ is an orthogonal projection in $\fH_n$ onto
a subspace $\fH^-_n$ iff
\[
P^-_0(n):=b_n(L)^{1/2}P^-_nb_n(L)^{-1/2}
\]
is an orthogonal projection in $\fH_0$ onto
a subspace
\[
\fH^-_0(n):=P^-_0(n)\fH_0=b_n(L)^{1/2}\fH^-_n\,.
\]
If this is the case, then
\[
P^+_0(n):=I_{\fH_0}-P^-_0(n)=
b_n(L)^{1/2}P^+_nb_n(L)^{-1/2}
\]
is an orthogonal projection in $\fH_0$ onto
a subspace
\[
\fH^+_0(n):=\fH^-_0(n)^{\bot_{\fH_0}}=
P^+_0(n)\fH_0=b_n(L)^{1/2}\fH^+_n\,.
\]
\end{lem}
\begin{proof}
Because
\[
P^-_0(n)^2=b_n(L)^{1/2}(P^-_n)^2b_n(L)^{-1/2}
\]
$P^-_0(n)$ is a projection iff so is $P^-_n$.

We show that the adjoint $P^-_0(n)^*$ of
$P^-_0(n)$ in $\fH_0$ is given by
\begin{equation}
P^-_0(n)^*=b_n(L)^{1/2}P^{-\,*}_nb_n(L)^{-1/2}
\label{eq:p1}
\end{equation}
on $\fH_0$, where $P^{-\,*}_n$ is the adjoint of
$P^-_n$ in $\fH_n$;
then it follows that $P^-_0(n)$ is self-adjoint
in $\fH_0$ iff so is $P^-_n$ in $\fH_n$:
The graph of the
adjoint $P^-_0(n)^*$ in $\fH_0$ consists of $(y,x)\in\fH^2_0$
such that $(\forall u\in\fH_0)$
\[
\braket{u,x}_0=\braket{P^-_0(n)u,y}_0\,.
\]
Every $u$ is of the form $u=b_n(L)^{1/2}f$
with some $f\in\fH_n$. Then
\[
\braket{u,x}_0=\braket{f,b_n(L)^{-1/2}x}_n
\]
and
\[
\braket{P^-_0(n)u,y}_0=\braket{P^-_nf,b_n(L)^{-1/2}y}_n
=\braket{f,P^{-\,*}_nb_n(L)^{-1/2}y}_n
\]
from which the claim follows.
The remaining statements are verified straightforwardly.
\end{proof}
The present lemma allows us to freely transfer
between the $\fH_n$-space representation and
the $\fH_0$-space representation.
In particular $\fH^-_0(0)=\fH^-_0$, but in general
$\fH^-_0(n)\neq \fH^-_0$ for $n\neq0$. The
equality holds for all $n$ iff
\begin{equation}
P^-_n=b_n(L)^{-1/2}P^-_0b_n(L)^{1/2}
\label{eq:Pn}
\end{equation}
on $\fH_n$; in this
case one would have $\fH^-_{n+l}=b_l(L)^{-1/2}\fH^-_n$
for $l\in\bbN_0$ (\cf Example~\ref{exam:pn}).
Moreover, $P^\pm_0(n)P^\mp_0(n+l)\neq0$ in general.
However, the product of projections vanishes
for $l\in2\bbZ$, provided that $P^-_{n+1}\subseteq P^-_n$;
see Lemma~\ref{lem:mn2l} below.

Let $n\in\bbZ$, $l\in\bbN_0$ as above and let
\[
\fH^-_{n,l}:=P^-_n\fH_{n+l}=\fH^-_n\mcap\fH_{n+l}\,.
\]
The second equality in $\fH^-_{n,l}$ is a particular
case of the following statement:
$P^-_n(\fH_n\mcap\fX)=\fH^-_n\mcap\fX$ for an arbitrary set $\fX$.
Indeed, the set $\fH^-_n\mcap\fX$ consists of $f\in\fH_{n}\mcap\fX$
such that $f\in P^-_n\fH_n$; using $f=f^-+f^+$
with $f^\pm:=P^\pm_nf$ one therefore gets that $f^+=0$ and
$f^-\in P^-_n(\fH_n\mcap\fX)$. For
$\fX=\fH_{n+l}$ $(\subseteq\fH_n)$ one deduces the above
equality as claimed.

Using the definition of the projection $P^-_0(n)$
it follows that
\[
\fH^-_{n,l}=b_n(L)^{-1/2}\fH^-_l(n)\,,\quad
\fH^-_l(n):=P^-_0(n)\fH_l=\fH^-_0(n)\mcap\fH_l
\]
and hence $\fH^-_l(n)$ is a subset of $\fH_l$.
Similarly one defines
$\fH^+_{n,l}:=P^+_n\fH_{n+l}$ and
$\fH^+_l(n):=P^+_0(n)\fH_l$. We note that
\[
P^s_0(n)P^{s^\prime}_0(n^\prime)=
P^{s^\prime}_0(n^\prime)P^s_0(n)\,,
\quad
s,s^\prime\in\{-,+\}\,,\quad
n,n^\prime\in\bbZ
\]
and that
\[
\fH^s_l(n)\mcap\fH^{s^\prime}_0(n^\prime)=
\fH^s_l(n)\mcap\fH^{s^\prime}_l(n^\prime)
\]
$l\in\bbN_0$. Both statements are shown by using
$P^-_0(n)\fX=\fH^-_0(n)\mcap\fX$, $\fX\subseteq\fH_0$.

In general $\fH^-_{n,l}\neq\fH^-_{n+l}$,
but the following holds.
\begin{lem}\label{lem:mainl}%
Let $n\in\bbZ$, $l\in\bbN_0$.
$\fH^-_{n,l}$ (resp. $\fH^-_l(n)$) is dense in $\fH^-_n$
(resp. $\fH^-_0(n)$).
Moreover $\fH^-_{n,l}=\fH^-_{n+l}$ iff
$\fH^-_{n+l}$ is dense in $\fH^-_n$, or equivalently,
iff $P^-_{n+l}\subseteq P^-_n$
(in fact, if $\fH^-_{n+l}\subseteq\fH^-_n$ densely,
then $\fH^+_{n+l}\subseteq\fH^+_n$ densely and
$P^\pm_{n+l}\subseteq P^\pm_n$; conversely, if
$P^-_{n+l}\subseteq P^-_n$, then $P^+_{n+l}\subseteq P^+_n$
and $\fH^\pm_{n+l}\subseteq\fH^\pm_n$ densely);
if this is the case then
\[
P^-_0(n+l)\subseteq
b_l(L)^{1/2}P^-_0(n)b_l(L)^{-1/2}
\]
and hence $\fH^-_0(n+l)=b_l(L)^{1/2}\fH^-_l(n)$
(and similarly for $P^+_0(n+l)$ and $\fH^+_0(n+l)$).
\end{lem}
\begin{proof}
The orthogonal complement $(\fH^-_{n,l})^{\bot_{\fH_n}}$
in $\fH_n$ of $\fH^-_{n,l}$ consists of all $g\in\fH_n$
such that $(\forall f\in\fH_{n+l})$
\[
0=\braket{P^-_nf,g}_n=\braket{f,P^-_ng}_n\,.
\]
Because $\fH_{n+l}$ is dense in $\fH_n$, this implies
$P^-_ng=0$; hence
$(\fH^-_{n,l})^{\bot_{\fH_n}}=\fH^+_n$.
This shows that $\fH^-_{n,l}\subseteq\fH^-_n$
densely in $\norm{\cdot}_n$-norm.
Similarly, the orthogonal complement
$\fH^-_{l}(n)^{\bot_{\fH_0}}$
in $\fH_0$ of $\fH^-_{l}(n)$ consists of all $v\in\fH_0$
such that $(\forall u^-\in\fH^-_{l}(n))$
$0=\braket{u^-,v}_0$. Now $u^-$ is of the form
$u^-=b_n(L)^{1/2}P^-_nf$ with some $f\in\fH_{n+l}$, so
\[
\braket{u,v}_0=\braket{b_n(L)^{1/2}P^-_nf,v}_0=
\braket{P^-_nf,b_n(L)^{-1/2}v}_n=
\braket{f,P^-_nb_n(L)^{-1/2}v}_n\,.
\]
This implies $P^-_nb_n(L)^{-1/2}v=0$, and hence
\[
\fH^-_{l}(n)^{\bot_{\fH_0}}=b_n(L)^{1/2}\fH^+_n=
\fH^+_0(n)\,.
\]
One concludes that $\fH^-_l(n)\subseteq\fH^-_0(n)$
densely in $\norm{\cdot}_0$-norm.

Next one shows that $\fH^-_{n+l}\subseteq\fH^-_n$
densely iff $P^-_{n+l}\subseteq P^-_n$.
The orthogonal complement $(\fH^-_{n+l})^{\bot_{\fH_n}}$
in $\fH_n$ of $\fH^-_{n+l}$ is the set of all
$g\in\fH_n$ such that $(\forall f\in\fH_{n+l})$
$0=\braket{P^-_{n+l}f,g}_n$.
If $P^-_{n+l}\subseteq P^-_n$, then one arrives at
the previously considered case, namely,
$(\fH^-_{n+l})^{\bot_{\fH_n}}=(\fH^-_{n,l})^{\bot_{\fH_n}}$;
hence $\fH^-_{n+l}\subseteq\fH^-_n$ densely.
Moreover,
$P^-_{n+l}\subseteq P^-_n$ implies that also
$\fH^+_{n+l}\subseteq\fH^+_n$ densely:
$(\fH^+_{n+l})^{\bot_{\fH_n}}$ is the set of all
$g\in\fH_n$ such that $(\forall f\in\fH_{n+l})$
\[
0=\braket{P^+_{n+l}f,g}_n=\braket{f,g}_n-
\braket{P^-_{n+l}f,g}_n\,;
\]
but
\[
\braket{P^-_{n+l}f,g}_n=
\braket{P^-_{n}f,g}_n=\braket{f,P^-_{n}g}_n
\]
so
\[
0=\braket{P^+_{n+l}f,g}_n=\braket{f,P^+_ng}_n\,.
\]
This shows $(\fH^+_{n+l})^{\bot_{\fH_n}}=\fH^-_n$.
Conversely,
$(\fH^-_{n+l})^{\bot_{\fH_n}}=\fH^+_{n}$ implies that
$(\forall f\in\fH_{n+l})$ $(\forall g\in\fH_n)$
\[
0=\braket{P^-_{n+l}f,P^+_ng}_n=
\braket{P^+_nP^-_{n+l}f,g}_n
\]
hence $P^+_nP^-_{n+l}=0$. On the other hand,
$(\fH^-_{n+l})^{\bot_{\fH_n}}=\fH^+_{n}$ also implies that
$(\fH^+_{n+l})^{\bot_{\fH_n}}=\fH^-_n$:
$(\fH^+_{n+l})^{\bot_{\fH_n}}$ is the set of all
$g\in\fH_n$ such that $(\forall f\in\fH_{n+l})$
\[
0=\braket{P^+_{n+l}f,g}_n=\braket{f,g}_n-
\braket{P^-_{n+l}f,g}_n\,;
\]
now
\[
\braket{P^-_{n+l}f,g}_n=
\braket{P^-_{n+l}f,P^-_ng}_n+
\braket{P^-_{n+l}f,P^+_ng}_n
\]
and
\[
\braket{P^-_{n+l}f,P^+_ng}_n=
\braket{P^+_nP^-_{n+l}f,g}_n=0
\]
so
\[
0=
\braket{P^+_{n+l}f,g}_n=\braket{f,g}_n-
\braket{P^-_{n+l}f,P^-_ng}_n
=\braket{P^-_{n+l}f,g}-\braket{P^-_{n+l}f,P^-_ng}_n\,.
\]
As a result
$(\fH^+_{n+l})^{\bot_{\fH_n}}$ is the set of all
$g\in\fH_n$ such that $(\forall f^-\in\fH^-_{n+l})$
$0=\braket{f^-,P^+_ng}_n$.
Because by hypothesis $\fH^-_{n+l}$ is dense in
$\fH^-_n$, this shows
$(\fH^+_{n+l})^{\bot_{\fH_n}}=\fH^-_n$, as claimed.
Sequentially,
$(\forall f\in\fH_{n+l})$ $(\forall g\in\fH_n)$
\[
0=\braket{P^+_{n+l}f,P^-_ng}_n=
\braket{P^-_nP^+_{n+l}f,g}_n
\]
and hence $P^-_nP^+_{n+l}=0$. This together with
$P^+_nP^-_{n+l}=0$ implies that
$P^\pm_{n+l}\subseteq P^\pm_n$.

If $P^-_{n+l}\subseteq P^-_{n}$ then
$\fH^-_{n,l}=\fH^-_{n+l}$ by definition. Assuming the converse,
again by definition one gets that
$P^-_{n}\fH_{n+l}=P^-_{n+l}\fH_{n+l}$, \ie
$P^-_{n}\vrt_{\fH_{n+l}}=P^-_{n+l}$. This shows that
$\fH^-_{n,l}=\fH^-_{n+l}$ iff
$\fH^-_{n+l}$ is dense in $\fH^-_n$, or equivalently,
iff $P^-_{n+l}\subseteq P^-_n$.

Using $P^-_{n+l}\subseteq P^-_{n}$, for $u\in\fH_0$
\begin{align*}
P^-_0(n+l)u=&
b_{n+l}(L)^{1/2}P^-_{n+l}b_{n+l}(L)^{-1/2}u=
b_{n+l}(L)^{1/2}P^-_{n}b_{n+l}(L)^{-1/2}u
\\
=&b_l(L)^{1/2}P^-_0(n)b_l(L)^{-1/2}u
\end{align*}
and this completes the proof of the lemma.
\end{proof}
\begin{exam}\label{exam}
Let $\mrm{H}^n=W^n_2(\bbR^\nu)$, $\nu\in\bbN$, be the Sobolev
space; then $\mrm{L}^2=\mrm{L}^2(\bbR^\nu)=\mrm{H}^0$.
Let $L$ such that
\[
\fH_n=b_n(L)^{-1/2}(\mrm{L}^2\ot\bbC^4)=\mrm{H}^n\ot\bbC^4\,,
\quad n\in\bbZ
\]
and
\[
P^-_n(\mrm{H}^n\ot\bbC^4)=\mrm{H}^n\ot\bbC^1=\fH^-_n\,,
\quad
P^+_n(\mrm{H}^n\ot\bbC^4)=\mrm{H}^n\ot\bbC^3=\fH^+_n\,.
\]
Then $P^-_{n+1}\subseteq P^-_{n}$, and similarly for
$P^+_n$. The subspaces
\[
\fH^-_0(n)=b_n(L)^{1/2}(\mrm{H}^n\ot\bbC^1)\,,
\quad
\fH^+_0(n)=b_n(L)^{1/2}(\mrm{H}^n\ot\bbC^3)\,.
\]
For $l\in\bbN_0$,
the subset
\[
\fH^-_{n,l}=(\mrm{H}^n\ot\bbC^1)\mcap
(\mrm{H}^{n+l}\ot\bbC^4)=\mrm{H}^{n+l}\ot\bbC^1=\fH^-_{n+l}
\]
is dense in $\fH^-_{n}$; and similarly for
$\fH^+_{n,l}=\fH^+_{n+l}\subseteq\fH^+_n$. Likewise, the subset
\begin{align*}
\fH^-_l(n)=&[b_n(L)^{1/2}(\mrm{H}^n\ot\bbC^1)]\mcap
(\mrm{H}^l\ot\bbC^4)
=b_n(L)^{1/2}[(\mrm{H}^n\ot\bbC^1)\mcap(\mrm{H}^{n+l}\ot\bbC^4)]
\\
=&b_n(L)^{1/2}(\mrm{H}^{n+l}\ot\bbC^1)=
b_l(L)^{-1/2}\fH^-_0(n+l)
\end{align*}
is dense in $\fH^-_0(n)$, and similarly for
$\fH^+_l(n)\subseteq\fH^+_0(n)$.
\end{exam}
Due to the dense inclusion $\fH_{n+1}\subseteq\fH_n$
one also has the following result.
\begin{lem}\label{lem:mn2l}
Assume that $P^-_{n+1}\subseteq P^-_n$ for all $n\in\bbZ$.
Then
\[
\fH^-_0=\fH^-_0(2n)\,,\quad
\fH^-_0(1)=\fH^-_0(2n+1)\,.
\]
\end{lem}
\begin{proof}
We show that $\fH^+_0(n)=\fH^+_0(n-2l)$ for
$n\in\bbZ$, $l\in\bbN_0$;
by relabeling
$n-2l$ by $n$, the result extends to all
$l\in\bbZ$. Taking the orthogonal complements one
deduces an analogous result for $\fH^-_0(n)$.

We use two facts: that $\fH^+_0(n)=\ker P^-_0(n)$ and
that $\fH_l\subseteq\fH_0$ densely for $l\in\bbN_0$.
The kernel of $P^-_0(n)$ consists of $u\in\fH_0$
such that $P^-_0(n)u=0$; this is equivalent to saying that
$(\forall v\in\fH_0)$ $\braket{v,P^-_0(n)u}_0=0$.
By Lemma~\ref{lem:mainl}
\[
P^-_0(n)u=b_l(L)^{1/2}P^-_0(n-l)b_l(L)^{-1/2}u
\Rightarrow
P^-_0(n-l)b_l(L)^{-1/2}u=0\,.
\]
Thus $(\forall v\in\fH_0)$
\begin{align*}
0=&\braket{v,P^-_0(n)u}_0=
\braket{v,P^-_0(n-l)b_l(L)^{-1/2}u}_0
=\braket{b_l(L)^{-1/2}P^-_0(n-l)v,u}_0
\\
=&\braket{P^-_0(n-2l)b_l(L)^{-1/2}v,u}_0
\quad(\text{by Lemma~\ref{lem:mainl}})\,.
\end{align*}
Since every $v$ is of the form
$v=b_l(L)^{1/2}w$ with some $w\in\fH_l$,
it follows that $(\forall w\in\fH_l)$
\[
0=\braket{P^-_0(n-2l)w,u}_0=
\braket{w,P^-_0(n-2l)u}_0\,.
\]
Since $\fH_l\subseteq\fH_0$ densely, the latter
implies that $P^-_0(n-2l)u=0$; hence
\[
\fH^+_0(n)=\ker P^-_0(n)=\ker P^-_0(n-2l)=
\fH^+_0(n-2l)
\]
as claimed.
\end{proof}
Thus, if the hypothesis of Lemma~\ref{lem:mn2l} holds,
then the projections $P^-_0(n)$, $n\in\bbZ$, are
in fact characterized by only two projections:
$P^-_0=P^-_0(2n)$ and $P^-_0(1)=P^-_0(2n+1)$;
in this case $P^-_n$ is as in \eqref{eq:Pn} for
$n\in2\bbZ$, and
\[
P^-_n=b_n(L)^{-1/2}P^-_0(1)b_n(L)^{1/2}
\]
for $n\in2\bbZ+1$.
But the converse is not necessarily true in general.
\begin{exam}\label{exam:pn}
Let $P^-_n$ be as in \eqref{eq:Pn}. Then
$P^-_0(n)=P^-_0$ for all $n\in\bbZ$. Let $l\in\bbN_0$;
then
\[
\fH^-_{n,l}:=P^-_n\fH_{n+l}=
b_n(L)^{-1/2}P^-_0b_n(L)^{1/2}\fH_{n+l}=
b_n(L)^{-1/2}\fH^-_{0,l}
\]
while
\[
\fH^-_{n+l}:=P^-_{n+l}\fH_{n+l}=
b_{n}(L)^{-1/2}\fH^-_l\,.
\]
Thus $\fH^-_{n,l}=\fH^-_{n+l}$ iff
\[
\fH^-_{0,l}(=P^-_0\fH_l)=\fH^-_l(=P^-_l\fH_l)
\]
or what is the same, iff $P^-_0\supseteq P^-_l$.
\end{exam}
\section{Projected operators}\label{sec:Lnpm}
Let $n\in\bbZ$. By scaling every self-adjoint operator $L_n$
in $\fH_n$ admits the form
\begin{equation}
L_n=b_n(L)^{-1/2}Lb_n(L)^{1/2}\,,\quad L=L_0
\label{eq:Ln1}
\end{equation}
on $\dom L_n=\fH_{n+2}$.
To every $L_n$ one associates densely defined
(Lemma~\ref{lem:mainl})
projected operators
\[
L^-_n:=P^-_nL_n\vrt_{\fH^-_{n,2}}\,,\quad
L^+_n:=P^+_nL_n\vrt_{\fH^+_{n,2}}
\]
in $\fH^-_n$ and $\fH^+_n$, respectively.
In analogy to \eqref{eq:Ln1}, every operator $L^-_n$
admits the form
\[
L^-_n=b_n(L)^{-1/2}L^-_0(n)b_n(L)^{1/2}\,,
\quad L^-_0(n):=P^-_0(n)L\vrt_{\fH^-_2(n)}
\]
and similarly for $L^+_n$. The operators $L^\pm_0(n)$
are considered in $\fH^\pm_0(n)$, and hence they are
densely defined.

Using $\fH^-_0(n):=P^-_0(n)\fH_0$ and $\fH_0=(L-z_1)\fH_2$,
$\fH^-_0(n)$ is the sum of sets
\begin{equation}
\fH^-_0(n)=\ran(L^-_0(n)-z_1)+P^-_0(n)L\fH^+_2(n)\,.
\label{eq:de}
\end{equation}
Thus in general the operator $L^-_0(n)-z_1$ is not
surjective (unlike $L-z_1$). But the following holds.
\begin{thm}\label{thm:surj}
Under hypothesis of Lemma~\ref{lem:mn2l}
the operator $L^-_0(n)-z_1$, $n\in\bbZ$, is surjective.
\end{thm}
\begin{proof}
By Lemma~\ref{lem:mainl}
\[
\ran(L^-_0(n)-z_1)=P^-_0(n)b_1(L)\fH^-_2(n)=
P^-_0(n)\fH^-_0(n+2)\,.
\]
Now apply Lemma~\ref{lem:mn2l}.
\end{proof}
The statement of the theorem is therefore
equivalent to the statement
\begin{equation}
P^-_0(n)L\fH^+_2(n)=\{0\}\,.
\label{eq:PPP}
\end{equation}
Indeed, by Lemmas~\ref{lem:mainl} and \ref{lem:mn2l}
\[
P^-_0(n)L\fH^+_2(n)=P^-_0(n)b_1(L)\fH^+_2(n)
=P^-_0(n)\fH^+_0(n+2)=P^-_0(n)\fH^+_0(n)=\{0\}
\]
so the sum in \eqref{eq:de} implies that the
operator $L^-_0(n)-z_1$ is surjective, and vice verse.
In this case the operators
$L^-_0(n)$ satisfy $L^-_0=L^-_0(2n)$ and
$L^-_0(1)=L^-_0(2n+1)$.
Analogous results hold for $L^+_0(n)$
and $L^\pm_n$.

If $L^{-\,*}_n$ is the adjoint in $\fH^-_n$ of
$L^-_n$ and if $L^-_0(n)^*$ is the adjoint in
$\fH^-_0(n)$ of $L^-_0(n)$, then
\begin{lem}\label{lem:corrLn}
$L^{-\,*}_n=b_n(L)^{-1/2}L^-_0(n)^*b_n(L)^{1/2}$.
\end{lem}
\begin{proof}
The basic arguments are as in the proof of \eqref{eq:p1}.
\end{proof}
\begin{thm}\label{thm:L-0n*}
Under hypothesis of Lemma~\ref{lem:mn2l}
the operator $L^-_0(n)$, $n\in\bbZ$, is self-adjoint in
$\fH^-_0(n)$.
\end{thm}
\begin{proof}
Consider the adjoint $L^-_0(n)^*$ as a linear relation
in $\fH^-_0(n)$. Then $L^-_0(n)^*$ consists of
$(y^-,x^-)\in\fH^-_0(n)^2$ such that
$(\forall w^-\in\fH^-_2(n))$
\[
\braket{w^-,x^-}_0=\braket{L^-_0(n)w^-,y^-}_0\,.
\]
Every $w^-\in\fH^-_2(n)$
is of the form $w^-=P^-_0(n)b_1(L)^{-1}v$ with some
$v\in\fH_0$. Then
\begin{align*}
\braket{L^-_0(n)w^-,y^-}_0=&
\braket{LP^-_0(n)b_1(L)^{-1}v,y^-}_0=
\braket{b_1(L)P^-_0(n)b_1(L)^{-1}v,y^-}_0
\\
&+\braket{P^-_0(n)b_1(L)^{-1}v,z_1y^-}_0
\\
=&\braket{b_1(L)P^-_0(n)b_1(L)^{-1}v,y^-}_0
+\braket{v,b_1(L)^{-1}z_1y^-}_0\,.
\end{align*}
By applying Lemma~\ref{lem:mainl}
\[
\braket{b_1(L)P^-_0(n)b_1(L)^{-1}v,y^-}_0=
\braket{P^-_0(n+2)v,y^-}_0=\braket{v,P^-_0(n+2)y^-}_0\,.
\]
On the other hand
\[
\braket{w^-,x^-}_0=
\braket{b_1(L)^{-1}v,x^-}_0=
\braket{v,b_1(L)^{-1}x^-}_0\,.
\]
Therefore $(y^-,x^-)\in\fH^-_0(n)^2$ such that
\[
b_1(L)^{-1}x^-=P^-_0(n+2)y^-+b_1(L)^{-1}z_1y^-\,.
\]
Because $y^-=u^-+u^+$ is the sum of disjoint elements
$u^\pm\in P^\pm_0(n+2)\fH^-_0(n)$ it follows from the above
that
\[
b_1(L)^{-1}x^-=u^-+b_1(L)^{-1}z_1(u^-+u^+)\,.
\]
Because $b_1(L)^{-1}\fH^-_0(n)=\fH^-_2(n-2)$ by
Lemma~\ref{lem:mainl},
from here one concludes that
\[
u^-\in\fH^-_2(n-2)\mcap P^-_0(n+2)\fH^-_0(n)=
\fH^-_2(n-2)\mcap\fH^-_2(n)\mcap\fH^-_2(n+2)\,.
\]
Sequentially
\[
x^-=b_1(L)u^-+z_1(u^-+u^+)=
P^-_0(n)b_1(L)u^-+z_1(u^-+u^+)=
L^-_0(n)u^-+z_1u^+\,.
\]
Finally,
by applying Lemma~\ref{lem:mn2l} one gets that
$u^-\in\fH^-_2(n)$ and $u^+=0$.
\end{proof}
\begin{cor}
$z_1\in\res L^-_0(n)$.
\end{cor}
\begin{proof}
This follows from Theorems~\ref{thm:surj}
and \ref{thm:L-0n*}.
\end{proof}
Under hypothesis of Lemma~\ref{lem:mn2l}
and applying Lemma~\ref{lem:corrLn},
the operator $L^-_n$ is therefore self-adjoint in $\fH^-_n$.
Moreover, $z_1\in\res L^-_n=\res L^-_0(n)$ or, what is
equivalent, $P^-_nL_n\fH^+_{n+2}=\{0\}$.
Similar conclusions apply to operators $L^+_0(n)$
and $L^+_n$.
\begin{lem}\label{lem:resLm0}
Under hypothesis of Lemma~\ref{lem:mn2l}
the resolvent
\[
(L^-_0(n)-z)^{-1}=P^-_0(n)(L-z)^{-1}P^-_0(n)\quad
\text{on}\quad\fH^-_0(n)
\]
for $z\in\res L\subseteq\res L^-_0(n)$
(and similarly for $L^+_0(n)$).
\end{lem}
\begin{proof}
First we derive the resolvent formula for
$z\in\res L\mcap\res L^-_0(n)$ and then we show
that $\res L\subseteq\res L^-_0(n)$.
Consider an arbitrary $v\in\fH_0$. Then,
for $z\in\res L$, $(\exists u\in\fH_2)$ $v=(L-z)u$.
Projecting the latter onto $\fH^-_0(n)$ and
applying \eqref{eq:PPP} yields
\[
P^-_0(n)v=(L^-_0(n)-z)P^-_0(n)u
\]
and the resolvent formula follows for
$z\in\res L\mcap\res L^-_0(n)$.

The eigenspace
\[
\fN_z(L^-_0(n))=\{u^-\in\fH^-_2(n)\vrt
P^-_0(n)(L-z)u^-=0\}
\]
is nontrivial for some $z\in\bbR$
(\cf Theorem~\ref{thm:L-0n*}).
From here and \eqref{eq:PPP} one gets that
\[
(L-z)u^-=P^+_0(n)(L-z)u^-=0\,;
\]
hence
\[
\fN_z(L^-_0(n))=\fH^-_2(n)\mcap\fN_z(L)\,.
\]
If $z\notin\sigma_p(L)$ then also
$z\notin\sigma_p(L^-_0(n))$, but the converse
$z\notin\sigma_p(L^-_0(n))$ implies only that
$\fN_z(L)=\fH^+_2(n)\mcap\fN_z(L)$ in this case.
Therefore $\sigma_p(L^-_0(n))\subseteq\sigma_p(L)$.

Now let $z\in\res L$; that is, $z\notin\sigma_p(L)$
and $\ran(L-z)=\fH_0$. Because by \eqref{eq:PPP}
\[
\ran(L-z)=\ran(L^-_0(n)-z)\dsum\ran(L^+_0(n)-z)
\]
it follows that
\[
\ran(L^-_0(n)-z)=\fH^-_0(n)\,,\quad
\ran(L^+_0(n)-z)=\fH^+_0(n)
\]
so $z\in\res L^-_0(n)$.
\end{proof}
Under the same hypothesis the resolvent of $L^-_n$ is given by
\[
(L^-_n-z)^{-1}=P^-_n(L_n-z)^{-1}P^-_n\quad
\text{on}\quad\fH^-_n
\]
for $z\in\res L_n=\res L$ (and similarly for $L^+_n$).

We summarize the main results obtained so far
in the following theorem.
\begin{thm}\label{thm:prn}
Let $\fH_{n+1}\subseteq\fH_n$ be the scale of Hilbert
spaces associated with a self-adjoint operator $L$ in
$\fH_0$. For each $n\in\bbZ$, let $P^-_n$ be an
orthogonal projection in $\fH_n$ onto a subspace
$\fH^-_n\subseteq\fH_n$; $\fH^+_n$ is the orthogonal
complement in $\fH_n$ of $\fH^-_n$. Assume that
$P^-_{n+1}\subseteq P^-_n$. Then the projections
$(P^-_n)_{n\in\bbZ}$ are characterized, by scaling,
by any two adjacent projections, say $P^-_0$ and $P^-_1$,
according to
\[
P^-_{2n}=b_n(L)^{-1}P^-_0b_n(L)\,,\quad
P^-_{2n+1}=b_n(L)^{-1}P^-_1b_n(L)\,.
\]
For each $n$, the subspace $\fH^-_n$ (resp. $\fH^+_n$)
is therefore a reducing subspace for the restriction $L_n$
to $\fH_{n+2}$ of $L$. The part of $L_n$ in $\fH^-_n$
(resp. $\fH^+_n$) is a self-adjoint operator.
\end{thm}
\begin{proof}
This follows from Lemmas~\ref{lem:mn2l},
\ref{lem:corrLn}, and Theorem~\ref{thm:L-0n*}.
\end{proof}
\section{Min-Max operators in a subspace}
In the present and subsequent paragraphs
$\fM^*_d\cG_{\mrm{A}}=\cG_{\mrm{A}}\fM_d$,
as in \eqref{eq:hypho2}, for an invertible
Hermitian $\cG_{\mrm{A}}$,
and $P^-_{n+1}\subseteq P^-_n$, $n\in\bbZ$,
as in Theorem~\ref{thm:prn}.
Let
\begin{align*}
A^\prime_{\min}:=&U_{\mrm{A}}A_{\min}U^{-1}_{\mrm{A}}
\\
=&\{\bigl((f^{\#},\xi),(Lf^\#,\fM_d\xi)\bigr)\vrt
f^{\#}\in\fH_{m+2}\,;\,
\xi\in\bbC^{md}\,;\,
\braket{\vp,f^\#}=[\cG_{\mrm{A}}\xi]_m\}\,.
\end{align*}
Then $A^\prime_{\min}$ is a closed, densely defined,
symmetric operator in $\cH^{\prime}_{\mrm{A}}$, whose
adjoint $A^{\prime\,*}_{\min}$ is given by
\begin{align*}
A^\prime_{\max}:=&A^{\prime\,*}_{\min}=
U_{\mrm{A}}A_{\max}U^{-1}_{\mrm{A}}
\\
=&\{\bigl((f^{\#}+h_{m+1}(c),\xi),
(Lf^\#+z_1h_{m+1}(c),\fM_d\xi+\eta(c))\bigr)\vrt
f^{\#}\in\fH_{m+2}\,;\,c\in\bbC^d\,;
\\
&\xi\in\bbC^{md}\}\,.
\end{align*}
If $(\bbC^d,\Gamma^{\mrm{A}}_0,\Gamma^{\mrm{A}}_1)$
is an OBT for $A_{\max}$ then the triple
$(\bbC^d,\Gamma^{\prime\,\mrm{A}}_0,\Gamma^{\prime\,\mrm{A}}_1)$,
with $\Gamma^{\prime\,\mrm{A}}_i:=\Gamma^{\mrm{A}}_iU^{-1}_{\mrm{A}}$,
$i\in\{0,1\}$, is an OBT for $A^\prime_{\max}$.

Let
\[
\Pi^\pm:=P^\pm_m\op I_{\bbC^{md}}\quad\text{in}\quad
\fH_m\op\bbC^{md}\,.
\]
Then $\Pi^-$ (resp. $\Pi^+$) is an orthogonal
(with respect to the $\fH_m\op\bbC^{md}$-metric)
projection onto a subspace
$\fH^-_m\op\bbC^{md}$ (resp. $\fH^+_m\op\bbC^{md}$).
Note that
\[
\Pi^-\Pi^+\neq0\,,\quad\Pi^+\Pi^-\neq0\,,\quad
\Pi^-+\Pi^+\neq I_{\fH_m\op\bbC^{md}}\,.
\]
However, given $\Pi^-$, the above inequalities
become the equalities with $\Pi^+$ replaced by
the orthogonal projection
$\Pi^{\prime\,+}:=I_{\fH_m\op\bbC^{md}}-\Pi^-$ onto
\[
(\fH^-_m\op\bbC^{md})^{\bot_{\fH_m\op\bbC^{md}}}
=\fH^+_m\op\{0\}\,.
\]
Likewise, given $\Pi^+$, the above inequalities
become the equalities with $\Pi^-$ replaced by
the orthogonal projection
$\Pi^{\prime\,-}:=I_{\fH_m\op\bbC^{md}}-\Pi^+$ onto
\[
(\fH^+_m\op\bbC^{md})^{\bot_{\fH_m\op\bbC^{md}}}
=\fH^-_m\op\{0\}\,.
\]

By Theorem~\ref{thm:prn}, $A^\prime_{\min}$ maps
\[
\dom A^\prime_{\min}\mcap(\fH^-_m\op\bbC^{md})=
\Pi^-\dom A^\prime_{\min}
\]
into $\fH^-_m\op\bbC^{md}$; therefore
$\fH^-_m\op\bbC^{md}$ is an invariant
(\cite[Definition~1.7]{Schmudgen12})
subspace for $A^\prime_{\min}$. Let $A^-_{\min}$ denote
the part of $A^\prime_{\min}$ in $\fH^-_m\op\bbC^{md}$,
that is
\begin{align*}
A^-_{\min}:=&
A^\prime_{\min}\vrt_{\Pi^-\dom A^\prime_{\min}}
=\Pi^-A^\prime_{\min}\vrt_{\Pi^-\dom A^\prime_{\min}}
\\
=&\{\bigl((f^{\#\,-},\xi),(L^-_mf^{\#\,-},\fM_d\xi)\bigr)\vrt
f^{\#\,-}\in\fH^-_{m+2}\,;\,
\xi\in\bbC^{md}\,;\,
\braket{\vp,f^{\#\,-}}=[\cG_{\mrm{A}}\xi]_m\}\,.
\end{align*}
Similarly one defines the part $A^+_{\min}$ of
$A^\prime_{\min}$ in $\fH^+_m\op\bbC^{md}$.
Because $\fH^+_m\op\{0\}$ (resp. $\fH^-_m\op\{0\}$)
is also an invariant subspace for $A^\prime_{\min}$,
the operator $A^\prime_{\min}$ is
represented by the orthogonal sum of its part
$A^-_{\min}$ in $\fH^-_m\op\bbC^{md}$
(resp. $A^+_{\min}$ in $\fH^+_m\op\bbC^{md}$)
and its part $L^+_{\min}\op0$ in
$\fH^+_m\op\{0\}$ (resp. $L^-_{\min}\op0$
in $\fH^-_m\op\{0\}$), where the operator
\[
L^+_{\min}:=
L^+_m\vrt_{\{f^{+}\in\fH^+_{m+2}\vrt
\braket{\vp,f^{+}}=0\}}\quad
(\text{resp.
$L^-_{\min}:=
L^-_m\vrt_{\{f^{-}\in\fH^-_{m+2}\vrt
\braket{\vp,f^{-}}=0\}}$})\,;
\]
symbolically ($[\op]$ indicates both
$\fH_m\op\bbC^{md}$-orthogonal and
$\cH^\prime_{\mrm{A}}$-orthogonal sum)
\begin{equation}
A^\prime_{\min}=A^-_{\min}[\op]
(L^+_{\min}\op0)
=
(L^-_{\min}\op0)[\op]
A^+_{\min}\,.
\label{eq:Amindec}
\end{equation}

Let $\vp^-$ (resp. $\vp^+$)
denote the vector valued functional
whose components $\vp^-_\sigma$ (resp. $\vp^+_\sigma$)
are defined by
\begin{align*}
&\vp^-_\sigma:=b_{m+2}(L)^{1/2}P^-_0(m)b_{m+2}(L)^{-1/2}
\vp_\sigma\in b_{m+2}(L)^{1/2}(\fH^-_0\setm\fH^-_1)
\\
&(\text{resp.
$\vp^+_\sigma:=b_{m+2}(L)^{1/2}P^+_0(m)b_{m+2}(L)^{-1/2}
\vp_\sigma\in b_{m+2}(L)^{1/2}(\fH^+_0\setm\fH^+_1)$})\,.
\end{align*}
The duality pairing $\braket{\vp^-_\sigma,\cdot}$
(resp. $\braket{\vp^+_\sigma,\cdot}$)
is defined via the $\fH_0$-scalar product in a usual way.
$\braket{\vp^-,\cdot}=(\braket{\vp^-_\sigma,\cdot})
\co\fH^-_{m+2}\lto\bbC^d$
denotes the action of the vector valued functional $\vp^-$,
and similarly for $\vp^+$.
\begin{lem}\label{lem:spir}
For $f^{\#\,-}\in\fH^-_{m+2}$
\[
\braket{\vp,f^{\#\,-}}=
\braket{\vp^-,f^{\#\,-}}=
\braket{h^-_{m+1},(L^-_m-z_1)f^{\#\,-}}_m
\]
and similarly for the action of $\vp$ on $\fH^+_{m+2}$.
\end{lem}
\begin{proof}
By the definition of the duality pairing and that of
$\vp^-_\sigma$
\begin{align*}
\braket{\vp^-_\sigma,f^{\#\,-}}=&
\braket{b_{m+2}(L)^{-1/2}\vp^-_\sigma,
b_{m+2}(L)^{1/2}f^{\#\,-}}_0
\\
=&\braket{P^-_0(m)b_{m+2}(L)^{-1/2}\vp_\sigma,
b_{m+2}(L)^{1/2}f^{\#\,-}}_0
\\
=&\braket{b_{m+2}(L)^{-1/2}\vp_\sigma,
P^-_0(m)b_{m+2}(L)^{1/2}f^{\#\,-}}_0\,.
\end{align*}
But
\[
b_{m+2}(L)^{1/2}f^{\#\,-}\in
b_{m+2}(L)^{1/2}\fH^-_{m+2}=
b_{m+2}(L)^{1/2}P^-_{m+2}\fH_{m+2}=
\fH^-_0(m+2)
\]
and hence by Lemma~\ref{lem:mn2l}
$b_{m+2}(L)^{1/2}f^{\#\,-}\in\fH^-_0(m)$;
therefore
\begin{align*}
\braket{b_{m+2}(L)^{-1/2}\vp_\sigma,
P^-_0(m)b_{m+2}(L)^{1/2}f^{\#\,-}}_0=&
\braket{b_{m+2}(L)^{-1/2}\vp_\sigma,
b_{m+2}(L)^{1/2}f^{\#\,-}}_0
\\
=&\braket{\vp_\sigma,f^{\#\,-}}\,.
\end{align*}
This proves the first equality. Using that
$b_{m+2}(L)^{1/2}f^{\#\,-}\in\fH^-_0(m)$, the second
equality is due to
\begin{align*}
\braket{h^-_{\sigma,m+1},(L^-_m-z_1)f^{\#\,-}}_m=&
\braket{h^-_{\sigma,m+1},b_1(L)f^{\#\,-}}_m
\\
=&\braket{b_m(L)^{1/2}P^-_mh_{\sigma,m+1},
b_{m+2}(L)^{1/2}f^{\#\,-}}_0
\\
=&\braket{P^-_0(m)b_{m+2}(L)^{-1/2}\vp_\sigma,
b_{m+2}(L)^{1/2}f^{\#\,-}}_0
\\
=&\braket{b_{m+2}(L)^{-1/2}\vp_\sigma,
b_{m+2}(L)^{1/2}f^{\#\,-}}_0=
\braket{\vp_\sigma,f^{\#\,-}}\,.
\end{align*}
The proof of $\braket{\vp,\cdot}$ on $\fH^+_{m+2}$
is analogous.
\end{proof}
By the lemma the boundary conditions defining the operators
$L^\pm_{\min}$ are therefore reduced to $\braket{\vp^\pm,f^\pm}=0$,
$f^\pm\in\fH_{m+2}$, where $\vp^-+\vp^+=\vp$. Explicitly
\[
L^-_{\min}:=
L^-_m\vrt_{\{f^{-}\in\fH^-_{m+2}\vrt
\braket{\vp^-,f^{-}}=0\}}\,,\quad
L^+_{\min}:=
L^+_m\vrt_{\{f^{+}\in\fH^+_{m+2}\vrt
\braket{\vp^+,f^{+}}=0\}}\,.
\]

Just like the functionals $\vp_\sigma$ define
the elements $h_{\sigma j}:=b_j(L)^{-1}\vp_\sigma$,
$j\in J$, that generate the linear space $\fK_{\mrm{A}}$,
the functionals $\vp^\pm_\sigma$ define the elements
\begin{equation}
h^\pm_{\sigma j}:=b_j(L)^{-1}\vp^\pm_\sigma=
P^\pm_{-m-2+2j}h_{\sigma j}
\label{eq:KApm}
\end{equation}
that generate (span) the linear subspaces
$\fK^\pm_{\mrm{A}}$ of $\fK_{\mrm{A}}$; that is,
$\fK_{\mrm{A}}=\fK^-_{\mrm{A}}\dsum \fK^+_{\mrm{A}}$.
The proof of the second equality in \eqref{eq:KApm}
uses the definition of $P^\pm_0(\cdot)$
and then Lemma~\ref{lem:mn2l}, in the same spirit
as in the proof of Lemma~\ref{lem:spir}.

Unlike the case of $A^\prime_{\min}$, the operator
$A^\prime_{\max}$ does not commute with the projection
$\Pi^-$ (resp. $\Pi^+$). The reason is that
now the projection of
$h_{m+1}(c)$ onto $\fH^-_m$ affects the value
of the extra term $\eta(c)\in\bbC^{md}$. This seems to
be better seen in the representation of the operator
$A^\prime_{\max}$
in the space $\fH_m\dsum\fK_{\mrm{A}}$, \ie in analyzing
the operator $A_{\max}$. Thus we have by Lemma~\ref{lem:1}
(here $k\in\fK_{\mrm{A}}$)
\[
A_{\max}(f^{\#}+h_{m+1}(c)+k)=L_{m-2}(f^{\#}+h_{m+1}(c))
+k^\prime\,,
\]
\[
k^\prime\in\fK_{\mrm{A}}\,,\quad
d(k^\prime)=\fM_dd(k)
\]
and
\[
L_{m-2}h_{m+1}(c)=z_1h_{m+1}(c)+h_m(c)
\]
where
\[
h_m(c)=b_1(L)h_{m+1}(c)=
\sum_{\alpha}
[\eta(c)]_\alpha h_\alpha=\sum_\sigma
c_\sigma h_{\sigma m}\in\fK_{\mrm{A}}\,.
\]
Now projecting $f^{\#}+h_{m+1}(c)+k$ onto
$\fH^-_m\dsum\fK_{\mrm{A}}$ one gets that
\begin{align*}
A_{\max}U^{-1}_{\mrm{A}}\Pi^-U_{\mrm{A}}
(f^{\#}+h_{m+1}(c)+k)=&
L^-_{m-2}(f^{\#\,-}+h^-_{m+1}(c))
+k^\prime
\\
=&L^-_mf^{\#\,-}+z_1h^-_{m+1}(c)+k^\prime+h^-_m(c)
\end{align*}
with
\[
h^-_m(c):=b_1(L)h^-_{m+1}(c)=
\sum_{\alpha}
[\eta(c)]_\alpha h^-_\alpha=
\sum_\sigma c_\sigma h^-_{\sigma m}\in\fK^-_{\mrm{A}}
\]
(it is precisely for this reason why $\eta(c)$ changes to
$\eta^-(c)\neq\eta(c)$; see below),
while
\begin{align*}
U^{-1}_{\mrm{A}}\Pi^-U_{\mrm{A}}A_{\max}
(f^{\#}+h_{m+1}(c)+k)=&
L^-_mf^{\#\,-}+z_1h^-_{m+1}(c)+k^\prime+h_m(c)
\\
=&A_{\max}U^{-1}_{\mrm{A}}\Pi^-U_{\mrm{A}}
(f^{\#}+h_{m+1}(c)+k)+h^+_m(c)
\end{align*}
with $h^+_m(c)\in\fK^+_{\mrm{A}}$ defined similarly
as $h^-_m(c)$. Because $h^\pm_m(c)\in\fK^\pm_{\mrm{A}}$
and $\fK^\pm_{\mrm{A}}\subseteq\fK_{\mrm{A}}$, it follows that
\[
h^\pm_m(c)=\sum_{\alpha}
[\eta(c)]_\alpha h^\pm_\alpha=
\sum_{\alpha}
[\eta^\pm(c)]_\alpha h_\alpha
\]
for $\eta^\pm(c)\in\bbC^{md}$ given by
\[
\eta^\pm(c)
:=\wtcG^{-1}_{\mrm{A}}\braket{h,h^\pm_m(c)}_{-m}
=\wtcG^{-1}_{\mrm{A}}\wtcG^\pm_{\mrm{A}}\eta(c)
\]
with the matrix
\[
\wtcG^\pm_{\mrm{A}}=
([\wtcG^\pm_{\mrm{A}}]_{\alpha\alpha^\prime})\,,
\quad
[\wtcG^\pm_{\mrm{A}}]_{\alpha\alpha^\prime}:=
\braket{h_\alpha,h^\pm_{\alpha^\prime}}_{-m}\,.
\]
With this notation, and going back to the representation
of $A_{\max}$ in $\fH_m\op\bbC^{md}$, one gets that
\[
A^\prime_{\max}\Pi^-(f^\#+h_{m+1}(c),\xi)=
(L^-_mf^{\#\,-}+z_1h^-_{m+1}(c),\fM_d\xi+
\eta^-(c))
\]
while
\[
\Pi^-A^\prime_{\max}(f^\#+h_{m+1}(c),\xi)=
(L^-_mf^{\#\,-}+z_1h^-_{m+1}(c),\fM_d\xi+\eta(c))\,.
\]
Similarly,
projecting $(f^{\#}+h_{m+1}(c),\xi)$ onto
$\fH^+_m\op\{0\}$ gives
\[
A^\prime_{\max}\Pi^{\prime\,+}(f^\#+h_{m+1}(c),\xi)=
(L^+_mf^{\#\,+}+z_1h^+_{m+1}(c),\eta^+(c))
\]
while
\[
\Pi^{\prime\,+}A^\prime_{\max}(f^\#+h_{m+1}(c),\xi)=
(L^+_mf^{\#\,+}+z_1h^+_{m+1}(c),0)\,.
\]
From these formulas one observes that
one still is able to represent the extension of the operator
$A^\prime_{\max}$ (but not the operator $A^\prime_{\max}$
itself) as the orthogonal sum of its parts
in subspaces $\fH^-_m\op\bbC^{md}$
(resp. $\fH^+_m\op\bbC^{md}$) and $\fH^+_m\op\{0\}$
(resp. $\fH^-_m\op\{0\}$),
similarly as in \eqref{eq:Amindec}, by
moving an element $(0,\eta^+(c))$ from
$A^\prime_{\max}\Pi^{\prime\,+}$ to
$A^\prime_{\max}\Pi^{-}$.

To make this precise, one
therefore introduces the linear relation
\begin{align*}
A^-_{\max}:=&
\{\bigl((f^{\#\,-}+h^-_{m+1}(c),\xi),
(L^-_mf^{\#\,-}+z_1h^-_{m+1}(c),\fM_d\xi+\eta(c))\bigr)\vrt
f^{\#\,-}\in\fH^-_{m+2}\,;
\\
&c\in\bbC^d\,;\,\xi\in\bbC^{md}\}
\end{align*}
in $\fH^-_m\op\bbC^{md}$ with the multivalued part
\[
\mul A^-_{\max}=\{0\}\times\eta^+(\Sigma^-)\,,\quad
\Sigma^-:=\{c\in\bbC^d\vrt\sum_\sigma c_\sigma
\vp^-_\sigma=0\}
\]
(the multivalued part is exactly the
orthogonal complement in $\cH^{\prime\,-}_{\mrm{A}}$
of $\dom A^-_{\min}$)
and the operator
\[
L^-_{\max}\op0=\Pi^{\prime\,-}A^\prime_{\max}
\vrt_{\Pi^{\prime\,-}\dom A^\prime_{\max}}
\]
in $\fH^-_m\op\{0\}$ with
\begin{align*}
L^-_{\max}:=&
\{(f^{\#\,-}+h^-_{m+1}(c),
L^-_mf^{\#\,-}+z_1h^-_{m+1}(c))\vrt
f^{\#\,-}\in\fH^-_{m+2}\,;\,c\in\bbC^d\}\,.
\end{align*}
Analogously one defines the linear relation $A^+_{\max}$ in
$\fH^+_m\op\bbC^{md}$, with the multivalued part
$\{0\}\times\eta^-(\Sigma^+)$,
and the operator $L^+_{\max}$ in $\fH^+_m$.
Note that the domain of the operator $L^-_{\max}$ in $\fH^-_m$
can be also written thus
\[
\dom L^-_{\max}=\fH^-_{m+2}\dsum\fN_z(L^-_{\max})\,,
\quad z\in\res L^-_m
\]
with the eigenspace
\[
\fN_z(L^-_{\max})=(L^-_m-z_1)(L^-_m-z)^{-1}
h^-_{m+1}(\bbC^d)
\]
and similarly for $L^+_{\max}$.
(The operators $L^\pm_{\max}$ should not be confused
with the triplet adjoint $L_{\max}$; as we show below,
$L^-_{\max}$ is the adjoint in $\fH^-_{m}$ of
$L^-_{\min}$, and similarly for $L^+_{\max}$.)

It follows from the above constructions that
the orthogonal (both in $\fH_m\op\bbC^{md}$-metric
and in $\cH^{\prime}_{\mrm{A}}$-metric)
componentwise sum of linear relations
(\cf \cite{Hassi12,Hassi09,Hassi07} for the notation)
\begin{equation}
A^-_{\max}[\hop](L^+_{\max}\op0)=
(L^-_{\max}\op0)[\hop]A^+_{\max}
\label{eq:Amaxdec}
\end{equation}
is an extension in $\fH_m\op\bbC^{md}$ of the
operator $A^\prime_{\max}$.
By comparing \eqref{eq:Amindec} with
\eqref{eq:Amaxdec} one concludes that
$A^-_{\min}\subseteq A^-_{\max}$
and $L^-_{\min}\subseteq L^-_{\max}$,
and similarly for $A^+_{\min}$ and $L^+_{\min}$.
In fact, one can say more.
\begin{thm}
The linear relation
$A^-_{\max}=A^{-\,*}_{\min}$ is the adjoint in
$\cH^{\prime\,-}_{\mrm{A}}$ of a nondensely defined
(in general), closed, symmetric operator $A^-_{\min}$.
\end{thm}
\begin{proof}
The main arguments are as in the proof of the
self-adjointness of $L^-_m$ (Theorem~\ref{thm:L-0n*})
by using in addition that the boundary condition
for $(f^{\#\,-},\xi)\in\dom A^-_{\min}$ implies that
$(\forall c\in\bbC^d)$
\begin{equation}
\braket{w,b_m(L)^{1/2}h^-_{m+1}(c)}_0=
\braket{\xi,\cG_{\mrm{A}}\eta(c)}_{\bbC^{md}}\,,
\quad
f^{\#\,-}=b_{m+2}(L)^{-1/2}P^-_0(m)w\,,
\label{eq:fm}
\end{equation}
$w\in\fH_0$;
note that
\[
b_m(L)^{1/2}h^-_{m+1}(c)=b_{m+2}(L)^{-1/2}\sum_\sigma
c_\sigma\vp^-_\sigma
\]
and the representation of $f^{\#\,-}$ is shown in
the proof of Lemma~\ref{lem:spir}.
The duality pairing then reads
\[
\braket{\vp^-,f^{\#\,-}}=
\braket{b_{m+2}(L)^{-1/2}\vp^-,b_{m+2}(L)^{1/2}f^{\#\,-}}_0
=\braket{b_{m+2}(L)^{-1/2}\vp^-,P^-_0(m)w}_0\,;
\]
but $b_{m+2}(L)^{-1/2}\vp^-\in\fH^-_0(m)$, so
the boundary condition reads
\[
\braket{\vp^-,f^{\#\,-}}=
\braket{b_{m+2}(L)^{-1/2}\vp^-,w}_0=[\cG_{\mrm{A}}\xi]_m
\]
from which \eqref{eq:fm} follows.

Now one computes $A^{-\,*}_{\min}$;
as a linear relation, it is the set of
$((y^-,\xi_y),(x^-,\xi_x))\in(\fH^-_m\op\bbC^{md})^2$
such that $(\forall(f^{\#\,-},\xi)\in\dom A^-_{\min})$
\begin{equation}
\braket{f^{\#\,-},x^-}_m+\braket{\xi,\cG_{\mrm{A}}\xi_x}_{\bbC^{md}}
=\braket{L^-_mf^{\#\,-},y^-}_m+
\braket{\fM_d\xi,\cG_{\mrm{A}}\xi_y}_{\bbC^{md}}\,.
\label{eq:ckdje}
\end{equation}
Applying the representation
\begin{align*}
x^-=&b_m(L)^{-1/2}u^-\,,\quad u^-\in\fH^-_0(m)\,,
\\
y^-=&b_m(L)^{-1/2}v^-\,,\quad v^-\in\fH^-_0(m)
\end{align*}
and using that $b_1(L)^{-1}\fH^-_0(m)=\fH^-_2(m)$
one gets that
\[
\braket{f^{\#\,-},x^-}_m=
\braket{w,b_1(L)^{-1}u^-}_0
\]
and
\begin{align*}
\braket{L^-_mf^{\#\,-},y^-}_m=&
\braket{b_1(L)f^{\#\,-},y^-}_m+
\braket{f^{\#\,-},z_1y^-}_m
\\
=&
\braket{w,v^-}_0+\braket{w,b_1(L)^{-1}z_1v^-}_0\,.
\end{align*}
Therefore \eqref{eq:ckdje} reads
\[
\braket{w,v^--b_1(L)^{-1}(u^--z_1v^-)}_0=
\braket{\xi,\cG_{\mrm{A}}(\xi_x-\fM_d\xi_y)}_{\bbC^{md}}\,.
\]
Comparing the latter with \eqref{eq:fm} yields
\begin{align*}
v^--b_1(L)^{-1}(u^--z_1v^-)=&b_m(L)^{1/2}h^-_{m+1}(c)\,,
\\
\xi_x=&\fM_d\xi_y+\eta(c)\,.
\end{align*}
The first equation above implies that
\[
v^--b_m(L)^{1/2}h^-_{m+1}(c)\in
b_1(L)^{-1}\fH^-_0(m)=\fH^-_2(m)
\]
that is
\[
y^-=f^-+h^-_{m+1}(c)\,,\quad f^-\in\fH^-_{m+2}\,.
\]
Then
\[
x^-=z_1y^-+b_1(L)f^-=L^-_mf^-+z_1h^-_{m+1}(c)\,.
\]
This proves $A^{-\,*}_{\min}=A^-_{\max}$.
It remains to verify that $A^-_{\min}$ is closed.
The adjoint $A^{-\,*}_{\max}$ consists of
$((y^-,\xi_y),(x^-,\xi_x))\in(\fH^-_m\op\bbC^{md})^2$
such that $(\forall f^{\#\,-}\in\fH^-_{m+2})$
$(\forall c\in\bbC^d)$ $(\forall\xi\in\bbC^{md})$
\begin{align*}
\braket{f^{\#\,-}+h^-_{m+1}(c),x^-}_m+
\braket{\xi,\cG_{\mrm{A}}\xi_x}_{\bbC^{md}}
=&\braket{L^-_mf^{\#\,-}+z_1h^-_{m+1}(c),y^-}_m
\\
+&
\braket{\fM_d\xi+\eta(c),\cG_{\mrm{A}}\xi_y}_{\bbC^{md}}\,.
\end{align*}
Using the representation of $f^{\#\,-}$, $x^-$, $y^-$
as above, and noting that
\[
\braket{h^-_{m+1}(c),x^-}_m=
\braket{c,\braket{h^-_{m+1},x^-}_m}_{\bbC^d}\,,
\quad
\braket{\eta(c),\cG_{\mrm{A}}\xi_y}_{\bbC^{md}}
=\braket{c,[\cG_{\mrm{A}}\xi_y]_m}_{\bbC^d}
\]
one gets that
\begin{align*}
0=&
\braket{w,v^--b_1(L)^{-1}(u^--z_1v^-)}_0
+
\braket{c,\braket{h^-_{m+1},z_1y^--x^-}_m+
[\cG_{\mrm{A}}\xi_y]_m}_{\bbC^d}
\\
&+
\braket{\xi,\cG_{\mrm{A}}(\fM_d\xi_y-\xi_x)}_{\bbC^{md}}
\end{align*}
and from which one concludes that
\[
v^-=b_1(L)^{-1}(u^--z_1v^-)\in\fH^-_2(m)\Rightarrow
x^-=L^-_my^-\,,\quad y^-\in\fH^-_{m+2}
\]
and
\[
\braket{h^-_{m+1},x^--z_1y^-}_m=
\braket{h^-_{m+1},(L^-_m-z_1)y^-}_m=
\braket{\vp,y^-}=[\cG_{\mrm{A}}\xi_y]_m
\]
(\cf Lemma~\ref{lem:spir}) and
$\xi_x=\fM_d\xi_y$. Thus $A^-_{\min}$ is closed,
and this completes the proof.
\end{proof}
The above proof also shows that:
\begin{cor}
The operator $L^-_{\max}=L^{-\,*}_{\min}$
is the adjoint in $\fH^-_m$ of a densely defined,
closed, symmetric operator $L^-_{\min}$.
\end{cor}
From here one concludes that $L^-_{\min}$
(resp. $L^+_{\min}$) is an essentially self-adjoint
operator in $\fH^-_0$ (resp. $\fH^+_0$).
Since $A^-_{\min}$ extends $L^-_{\min}$ to
$\cH^{\prime\,-}_{\mrm{A}}$ just like $A_{\min}$
extends $L_{\min}$ to $\cH_{\mrm{A}}$
it is therefore a subject of
interest to formulate a similar realization theorem in
the A-model for the symmetric operator $L^-_{\min}$.
This is done in the next (the last) paragraph.
\section{Realization theorem in a subspace}
By a straightforward computation and applying
Lemma~\ref{lem:spir}, the boundary
form of the linear relation $A^-_{\max}$ is given by
\begin{align*}
&[(f^{\#\,-}+h^-_{m+1}(c),\xi),
(L^-_mg^{\#\,-}+z_1h^-_{m+1}(c^\prime),
\fM_d\xi^\prime+\eta(c^\prime))]^\prime_{\mrm{A}}
\\
&-[(L^-_mf^{\#\,-}+z_1h^-_{m+1}(c),\fM_d\xi+\eta(c)),
(g^{\#\,-}+h^-_{m+1}(c^\prime),\xi^\prime)]^\prime_{\mrm{A}}
\\
&=\braket{c,
\braket{\vp^-,g^{\#\,-}}-[\cG_{\mrm{A}}\xi^\prime]_m
}_{\bbC^d}-
\braket{\braket{\vp^-,f^{\#\,-}}-[\cG_{\mrm{A}}\xi]_m,
c^\prime}_{\bbC^d}
\end{align*}
for $f^{\#\,-},g^{\#\,-}\in\fH^-_{m+2}$;
$c,c^\prime\in\bbC^d$; $\xi,\xi^\prime\in\bbC^{md}$.
By introducing the mappings from
$A^-_{\max}$ to $\bbC^d$ by
\begin{equation}
\Gamma^{\mrm{A}\,-}_0\whf^-:=c\,,\quad
\Gamma^{\mrm{A}\,-}_1\whf^-:=
\braket{\vp^-,f^{\#\,-}}-[\cG_{\mrm{A}}\xi]_m\,,
\label{eq:OBTAm}
\end{equation}
\[
\whf^-=\bigl((f^{\#\,-}+h^-_{m+1}(c),\xi),
(L^-_mf^{\#\,-}+z_1h^-_{m+1}(c),\fM_d\xi+\eta(c))\bigr)
\in A^-_{\max}
\]
the above boundary form simplifies thus
\[
[f^-,g^{\prime\,-}]^\prime_{\mrm{A}}-
[f^{\prime\,-},g^-]^\prime_{\mrm{A}}=
\braket{\Gamma^{\mrm{A}\,-}_0\whf^-,
\Gamma^{\mrm{A}\,-}_1\whg^-}_{\bbC^d}-
\braket{\Gamma^{\mrm{A}\,-}_1\whf^-,
\Gamma^{\mrm{A}\,-}_0\whg^-}_{\bbC^d}\,,
\]
\[
\whf^-=(f^-,f^{\prime\,-})\in A^-_{\max}\,,\quad
\whg^-=(g^-,g^{\prime\,-})\in A^-_{\max}
\]
and it therefore represents the Green identity.
Consider $\Gamma^{\mrm{A}\,-}\co\whf^-\lmap
(\Gamma^{\mrm{A}\,-}_0\whf^-,\Gamma^{\mrm{A}\,-}_1\whf^-)$
from $A^-_{\max}$ to $\bbC^d\times\bbC^d$ as an
(isometric) linear
relation from $(\cH^{\prime\,-}_{\mrm{A}})^2$ to
$\bbC^d\times\bbC^d$. Thus by definition
$\dom \Gamma^{\mrm{A}\,-}=A^-_{\max}$
and
$\ker \Gamma^{\mrm{A}\,-}=A^-_{\min}$. Moreover,
the multivalued part $\mul \Gamma^{\mrm{A}\,-}$
consists of $(c,0)$ such that
$c\in\Sigma^-\mcap\Sigma^+=\{0\}$; hence
$\Gamma^{\mrm{A}\,-}$ is an operator.
Below we show that $\Gamma^{\mrm{A}\,-}$ is
a unitary relation from $(\cH^{\prime\,-}_{\mrm{A}})^2$ to
$\bbC^d\times\bbC^d$
(by the above, it would actually suffice to show
that $\dom(\Gamma^{\mrm{A}\,-})^{[+]}=\ran\Gamma^{\mrm{A}\,-}$).
By
\cite[Corollary~2.4(iii)]{Derkach06} this would imply that
$\Gamma^{\mrm{A}\,-}$ is surjective, and that therefore
the triple $(\bbC^d,\Gamma^{\mrm{A}\,-}_0,\Gamma^{\mrm{A}\,-}_1)$
is an OBT for $A^-_{\max}$.
\begin{lem}
$(\bbC^d,\Gamma^{\mrm{A}\,-}_0,\Gamma^{\mrm{A}\,-}_1)$
is an OBT for $A^-_{\max}$.
\end{lem}
\begin{proof}
By definition, the Krein space adjoint
$(\Gamma^{\mrm{A}\,-})^{[+]}$ is a linear relation
consisting of
\[
\left((\chi,\chi^\prime),\left((y^-,\xi_y),(x^-,\xi_x)\right)
\right)\in\bbC^{2d}\times(\fH^-_m\op\bbC^{md})^2
\]
such that
$(\forall f^{\#\,-}\in\fH^-_{m+2})$
$(\forall c\in\bbC^d)$ $(\forall\xi\in\bbC^{md})$
\begin{align*}
&\braket{f^{\#\,-}+h^-_{m+1}(c),x^-}_m+
\braket{\xi,\cG_{\mrm{A}}\xi_x}_{\bbC^{md}}
\\
&-
\braket{L^-_mf^{\#\,-}+z_1h^-_{m+1}(c),y^-}_m-
\braket{\fM_d\xi+\eta(c),\cG_{\mrm{A}}\xi_y}_{\bbC^{md}}
\\
&=\braket{c,\chi^\prime}_{\bbC^d}-
\braket{\braket{h^-_{m+1},(L^-_m-z_1)f^{\#\,-}}_m-
[\cG_{\mrm{A}}\xi]_m,\chi}_{\bbC^d}\,.
\end{align*}
The above equation splits into three equations
\[
(\forall f^{\#\,-})\;
\braket{f^{\#\,-},x^--z_1h^-_{m+1}(\chi)}_m=
\braket{L^-_mf^{\#\,-},y^--h^-_{m+1}(\chi)}_m\,,
\]
\[
(\forall c)\;0=
\braket{c,\braket{h^-_{m+1},x^--z_1y^-}_m-
[\cG_{\mrm{A}}\xi_y]_m-\chi^\prime}_{\bbC^d}\,,
\]
\[
(\forall \xi)\;0=
\braket{\xi,\cG_{\mrm{A}}(\xi_x-
\fM_d\xi_y-\eta(\chi))}_{\bbC^{md}}\,.
\]
Because $L^-_m$ is self-adjoint in $\fH^-_m$,
the first equation gives
\[
y^-=f^-+h^-_{m+1}(\chi)\,,\quad f^-\in\fH^-_{m+2}\,,
\quad
x^-=L^-_mf^-+z_1h^-_{m+1}(\chi)\,.
\]
Then the second equation yields
\[
\chi^\prime=\braket{\vp^-,f^-}-[\cG_{\mrm{A}}\xi_y]_m
\quad(\text{Lemma~\ref{lem:spir}})\,.
\]
Finally, by the third equation
\[
\xi_x=\fM_d\xi_y+\eta(\chi)\,.
\]
As a result $(\Gamma^{\mrm{A}\,-})^{[+]}=
(\Gamma^{\mrm{A}\,-})^{-1}$.
\end{proof}
Let
\begin{equation}
\Gamma^{-}_{0}(f^{\#\,-}+h^-_{m+1}(c)):=c\,,
\quad
\Gamma^{-}_{1}(f^{\#\,-}+h^-_{m+1}(c)):=
\braket{\vp^-,f^{\#\,-}}
\label{eq:OBTAm2}
\end{equation}
for $f^{\#\,-}+h^-_{m+1}(c)\in\dom L^-_{\max}$.
The above proof also shows that:
\begin{cor}
$(\bbC^d,\Gamma^{-}_{0},\Gamma^{-}_{1})$ is
an OBT for $L^-_{\max}$.
\end{cor}
We are now ready to state the main realization theorem
in the A-model for the symmetric operator $L^-_{\min}$,
by assuming \eqref{eq:hypho2} and $P^-_{n+1}\subseteq P^-_n$,
$n\in\bbZ$.
\begin{thm}\label{thm:mn2}
The extensions to $\cH^{\prime\,-}_{\mrm{A}}$
of a densely defined, closed, symmetric operator
$L^-_{\min}=L_{\min}\mcap(\fH^-_m)^2$
in $\fH^-_m$, which has defect numbers $(d,d)$
and which is essentially
self-adjoint in $\fH^-_0$, are described by
the proper extensions in $\cH^{\prime\,-}_{\mrm{A}}$
of a nondensely defined (in general), closed,
symmetric operator $A^-_{\min}=
A^\prime_{\min}\mcap(\fH^-_m\op\bbC^d)^2$.
A proper extension $A^-_\Theta$
is characterized by restricting the adjoint linear relation
$A^-_{\max}=A^{-\,*}_{\min}$ in $\cH^{\prime\,-}_{\mrm{A}}$
to the set of $\whf^-\in A^-_{\max}$ such that the pair
$(\Gamma^{\mrm{A}\,-}_0\whf^-,\Gamma^{\mrm{A}\,-}_1\whf^-)$
is an element of a linear relation $\Theta$ in $\bbC^d$;
an OBT $(\bbC^d,\Gamma^{\mrm{A}\,-}_0,\Gamma^{\mrm{A}\,-}_1)$
for $A^-_{\max}$ is as in \eqref{eq:OBTAm}.
The Krein--Naimark resolvent formula for
a (closed) proper extension $A^-_\Theta$ reads
\[
(A^-_\Theta-z)^{-1}=
(A^-_0-z)^{-1}+\gamma^-_{\mrm{A}}(z)
(\Theta-M^-_{\mrm{A}}(z))^{-1}\gamma^-_{\mrm{A}}(\ol{z})^*
\]
for $z\in\res A^-_0\mcap\res A^-_\Theta$.
A distinguished self-adjoint extension $A^-_0$ of $A^-_{\min}$
is a self-adjoint operator $A^-_0:=A^-_{\{0\}\times\bbC^d}$
whose resolvent is given by
\[
(A^-_0-z)^{-1}=(L^-_m-z)^{-1}\op
(\fM_d-z)^{-1}
\]
for $z\in\res A^-_0=\res L^-_m\setm\{z_1\}$.
The $\gamma$-field $\gamma^-_{\mrm{A}}$ and the Weyl function
$M^-_{\mrm{A}}$ corresponding to
$(\bbC^d,\Gamma^{\mrm{A}\,-}_0,\Gamma^{\mrm{A}\,-}_1)$
are described by
\[
\gamma^-_{\mrm{A}}(z)=
\left((L^-_m-z_1)(L^-_m-z)^{-1}h^-_{m+1}(\cdot),
-(\fM_d-z)^{-1}\eta(\cdot)\right)
\quad\text{on}\quad\bbC^d\,,
\]
\[
M^-_{\mrm{A}}(z)=q^-(z)+r(z)\quad\text{on}\quad\bbC^d
\]
for $z\in\res A^-_0$. The matrix valued function $q^-$
given by
\begin{align*}
q^-(z)=&([q^-(z)]_{\sigma\sigma^\prime})\in[\bbC^d]\,,
\quad z\in\res L^-_m\,,
\\
[q^-(z)]_{\sigma\sigma^\prime}:=&
(z-z_1)
\braket{\vp^-_\sigma,
(L^-_m-z)^{-1}h^-_{\sigma^\prime,m+1}}
\end{align*}
is the Weyl function which corresponds
to the OBT $(\bbC^d,\Gamma^{-}_{0},\Gamma^{-}_{1})$,
\eqref{eq:OBTAm2}, for the adjoint operator
$L^-_{\max}=L^{-\,*}_{\min}$ in $\fH^-_m$.
\end{thm}
\begin{proof}
In view of what has been achieved so far,
it remains to compute the $\gamma$-field and
the Weyl function. But these functions follow
straightforwardly from their definitions as long
as one notices that the eigenspace of $A^-_{\max}$
for the eigenvalue $z\in\res L^-_m\setm\{z_1\}$
consists of $(f^{\#\,-}+h^-_{m+1}(c),\xi)\in\dom A^-_{\max}$
such that
\[
f^{\#\,-}=(z-z_1)(L^-_m-z)^{-1}h^-_{m+1}(c)\,,\quad
\xi=-(\fM_d-z)^{-1}\eta(c)\,.
\]
Because $L^-_{\max}=A^-_{\max}\mcap(\fH^-_m\op\{0\})^2$,
the results for $L^-_{\max}$ are derived analogously.
\end{proof}
In particular, putting $P^-_n=I_{\fH_n}$ (hence $P^+_n=0$),
$n\in\bbZ$, the part of the theorem concerning the
Weyl function $q^-$ yields the following:
\begin{cor}\label{cor:mn2x}
The Krein $Q$-function $q$ is the Weyl function
associated with the OBT
$(\bbC^d,\Gamma_{0},\Gamma_{1})$,
\[
	\Gamma_{0}(f^\#+h_{m+1}(c)):=c\,,\quad
	\Gamma_{1}(f^\#+h_{m+1}(c)):=
	\braket{\vp,f^\#}
\]
$(f^\#\in\fH_{m+2},c\in\bbC^d)$,
for the adjoint $L^{*}_{\min}$ of $L_{\min}$ in $\fH_m$.
The domain
$\dom L^{*}_{\min}=\fH_{m+2}\dsum\fN_z(L^{*}_{\min})$,
where the eigenspace
$\fN_z(L^{*}_{\min})=(L-z)^{-1}h_m(\bbC^d)$,
$z\in\res L$.\qed
\end{cor}
An analogous theorem can be formulated for
$L^+_{\min}$ as well, where the corresponding
Weyl function $M^+_{\mrm{A}}=q^++r$ is the sum
of the Weyl function $q^+$ of $L^+_{\min}$ and the
generalized Nevanlinna function $r$.

Let
\[
\whh_\sigma:=b_{m+2}(L)^{-1/2}\vp_\sigma
\in\fH_0\setm\fH_1\,.
\]
Using this definition and the operator identity
\[
(L-z_1)(L-z)^{-1}=I_{\fH_0}+(z-z_1)(L-z)^{-1}
\]
the Weyl function
$q$ is rewritten in terms of the initial operator
$L$ and the reference $\fH_0$-scalar product according to
\[
[q(z)]_{\sigma\sigma^\prime}=
(z-z_1)\braket{\whh_\sigma,\whh_{\sigma^\prime}}_0+
(z-z_1)^2
\braket{\whh_\sigma,(L-z)^{-1}\whh_{\sigma^\prime}}_0\,,
\]
$z\in\res L$.
Using in addition \eqref{eq:PPP} and applying
\cite[Proposition~5.26]{Schmudgen12} and
Lemma~\ref{lem:resLm0}, the Weyl function $q^-$
admits the form
\[
[q^-(z)]_{\sigma\sigma^\prime}=
(z-z_1)\braket{\whh_\sigma,P^-_0(m)\whh_{\sigma^\prime}}_0+
(z-z_1)^2
\braket{\whh_\sigma,P^-_0(m)
(L-z)^{-1}\whh_{\sigma^\prime}}_0\,,
\]
$z\in\res L$, and similarly for $q^+$.
Thus the Weyl function $q=q^-+q^+$
of the symmetric operator $L_{\min}$
is the sum of the Weyl functions $q^\pm$
of the corresponding symmetric restrictions $L^\pm_{\min}$.
The latter property of additivity is clearly a consequence
of the initial hypothesis that the subspaces
$\fH^\pm_0$ reduce the operator $L$ (Theorem~\ref{thm:prn}).
\section*{Acknowledgement}
The author acknowledges the referees for their critique
and valuable remarks.

\end{document}